\documentclass[11pt]{article}

\usepackage{amsfonts,amssymb,amsopn,amsmath,mathrsfs,theorem,bbm}
\usepackage[round]{natbib}
\usepackage{hyperref}
\usepackage{graphicx}
\usepackage{verbatim}
\usepackage{authblk}
\usepackage{enumerate}

\newcounter{subsection1}[section]
\newtheorem{defn}[subsection1]{Definition}
\newtheorem{lemma}[subsection1]{Lemma}
\newtheorem{prop}[subsection1]{Proposition}
\newtheorem{theorem}[subsection1]{Theorem}
\newtheorem{remark}[subsection1]{Remark}
\newtheorem{cor}[subsection1]{Corollary}

\newcounter{assumptions}
\newenvironment{proof}{\vspace{1ex}\noindent{\textsc{Proof:}}\hspace{0.5em}}{\hfill\qed\vspace{1ex}}

\numberwithin{equation}{section} 
\numberwithin{subsection1}{section}

\begin{document}

\def\comment{\bfseries\textsc}

\def\ra{\Rightarrow} 
\def\to{\rightarrow} 
\def\iff{\Leftrightarrow}
\def\sw{\subseteq} 
\def\mc{\mathcal} 
\def\mb{\mathbb} 
\def\sc{\setminus} 
\def\p{\partial} 
\def\v{\mathbf} 
\def\E{\mb{E}} 
\def\P{\mb{P}}
\def\R{\mb{R}} 
\def\IR{\mb{R}} 
\def\C{\mb{C}} 
\def\N{\mb{N}}
\def\Q{\mb{Q}}
\def\Z{\mb{Z}}
\def\B{\mb{B}}
\def\~{\sim}
\def\-{\,;\,} 
\def\com{\leftrightarrow}
\def\|{\,|\,} 
\def\li{\langle}
\def\ri{\rangle}
\def\wt{\widetilde}
\def\qed{$\blacksquare$}
\def\1{\mathbbm{1}}
\def\cadlag{c\`{a}dl\`{a}g}
\def\slfv{S$\Lambda$FV}
\def\slfvs{S$\Lambda$FVS}
\def\l{\left}
\def\r{\right}
\def\CC{{C\nolinebreak[4]\hspace{-.05em}\raisebox{.4ex}{\tiny\bf ++}}}

\newcommand \ind {\mathbf{1}}

\author[3]{Alison Etheridge\thanks{etheridg@stats.ox.ac.uk, supported in part by EPSRC Grant EP/I01361X/1}}
\author[1,2]{Nic Freeman\thanks{n.p.freeman@sheffield.ac.uk}}
\author[3]{Daniel Straulino\thanks{Supported by CONACYT}}

\affil[1]{Heilbronn Institute for Mathematical Research, University of Bristol}
\affil[2]{School of Mathematics, University of Bristol}
\affil[3]{Department of Statistics, University of Oxford}

\title{The Brownian Net and 
Selection in the Spatial $\Lambda$-Fleming-Viot Process}
\date{\today}
\maketitle

\begin{abstract}
We obtain the Brownian net of \cite{SS2008} as the scaling limit of the paths traced out by a system of continuous (one-dimensional)
space and time branching and coalescing random walks. This demonstrates a certain universality of the net, which we 
have not seen explored elsewhere. The walks themselves arise in a natural way as the ancestral lineages relating individuals
in a sample from a biological population evolving according to the spatial Lambda-Fleming-Viot process. Our scaling reveals the 
effect, in dimension one, of spatial structure on the spread of a selectively advantageous gene through such a population.
\end{abstract}



\section{Introduction}\label{intro}

The {\em Brownian net}, introduced by \cite{SS2008}, arises as the scaling limit of 
a system of branching and coalescing random walk paths. It extends, in a natural way, the 
{\em Brownian web}, which originated in the work
of \cite{arratia:1979}. In the Brownian web there is no branching. It can be thought of as the
diffusive limit of a system of one-dimensional coalescing random walk paths, one started 
from each point of the space-time (diamond) lattice. Informally, the web is then a system of coalescing
Brownian paths, one started from each space-time point.
\cite{FIN2004} formulated the Brownian web
as a random variable taking its values in the space of compact sets of paths, 
equipped with a topology under which it is a Polish space. In this framework, the powerful techniques
of weak convergence become available and as a result the Brownian web emerges as the limit of 
a wide variety of one-dimensional coalescing systems; 
e.g.~\cite{FFW2005},~\cite{newman/ravishankar/sun:2005}. 
This points to a certain
`universality' of the Brownian web.

In the Brownian net, each path has a small probability (tending to zero in the scaling limit) 
of branching in each time step. The limiting object is (even) more difficult to visualise than 
the Brownian web, as there will be a multitude of paths emanating from each space-time point.
Nonetheless, \cite{SS2008} show that it can be characterised through 
the systems of `left-most' and `right-most' paths from each point, each of which itself forms a Brownian web (with drift).
Motivated by the study of perturbations of one-dimensional voter models, 
\cite{newman/ravishankar/schertzer:2015} show that, by starting from systems of 
random walks that branch, coalesce and die on the diamond lattice,
the Brownian net can be extended still further to 
include a killing term. However, we have not seen the `universality' 
of the Brownian net explored.
Our main result, Theorem~\ref{result dimension one}, is a contribution in this direction. It
establishes an appropriate scaling under which the paths traced out by a system of
branching and coalescing {\em continuous time and space} random walks in
one spatial dimension converges
to the Brownian net.

The original motivation for our work was a question of interest in population genetics: 
when will the action of natural selection on a gene in a spatially structured population 
cause a detectable trace in the patterns of genetic variation observed in the contemporary population?
We deal with the most biologically interesting case of a population evolving in a two-dimensional 
spatial continuum in \cite{EFP2016}. 
Our work in this paper uncovers some of the rich mathematical structure 
underlying mathematical models 
for biological populations evolving in one-dimensional spatial continua.
In particular, we study the systems of interacting random walks that, 
as dual processes (corresponding to ancestral lineages of the model), describe the relationships, 
across both time and space, between individuals sampled from those populations. 

It is
natural to ask whether the model of \cite{newman/ravishankar/schertzer:2015} has a biological
interpretation. It does: killing corresponds to a mutation term. This was 
observed by \cite{SS2008} (c.f.~\cite{etheridge/wang/yu:2015}).
However, in view of the technical challenges
to be overcome to handle the additional killing term,
even on a diamond lattice, we do not explore this further here.

Our starting point will be the Spatial $\Lambda$-Fleming-Viot process with selection ({\slfvs}) which 
(along with its dual) was introduced and constructed in \cite{EVY2014}. The dynamics of both the {\slfvs} 
and its dual are driven by a Poisson Point Process of {\em events} (which model reproduction in the population)
and will be described in detail in Section~\ref{slfvs}. 
Roughly, each event prescribes a region in which reproduction takes place. A 
proportion $\upsilon$ of the population in the affected region is replaced by
offspring of a single parent. We shall refer to $\upsilon$ as the impact
of the event.
In the absence of selection, the dual process of ancestral lineages
is a modification of the 
`Poisson trees' of \cite{FFW2005}. 
With selection, our dual follows `potential' ancestral lineages, which
introduces a branching 
mechanism, with the rate of branching determined by the presence of lineages 
in a region, but not increasing with their density. 
Our main result, Theorem~\ref{result dimension one}, is that in one spatial
dimension and when the impact $\upsilon=1$ (which prevents ancestral lineages
from jumping over one another), when suitably scaled the system of 
branching and coalescing ancestral lineages converges to the Brownian net.

Without selection, the corresponding objects converge (after
scaling) to the Brownian web. In that setting, we believe (and \cite{berestycki/etheridge/veber:2013} 
provides
strong supporting evidence) that the random walks can even be allowed to jump over one another and the
only effect on the limiting object is a simple 
scaling of time (given by one minus the crossing probability of `nearby' paths). This would mirror the results of
\cite{newman/ravishankar/sun:2005}, in which systems of coalescing non-simple random walks with
crossing paths are shown to converge to the Brownian web.
When we try to include selection in this limit, allowing paths to cross has a more complicated effect, as we illustrate 
through simulations in Section~\ref{numerics}. It is an intriguing open question to explain the pictures 
that we present there.

In \cite{EVY2014}, scaling limits of the {\slfvs} were considered in which the local 
population density tends to infinity. In that case, the classical Fisher-KPP equation and its stochastic analogue can be recovered. The dual process of branching and coalescing 
lineages converges to branching Brownian motion, with coalescence of lineages at a rate determined by 
the local time that they spend together. 
In this article we are interested in a very different regime, in which coalescence of lineages is instantaneous
on meeting. 

Although our result owes a lot 
to the existing literature, the continuum setting introduces some new features. In particular,
some care is needed in extending the self-duality of the systems of branching and coalescing simple
random walks that appear in \cite{SS2008} to an `approximate' self-duality of the {\cadlag} walks in continuous time
and space that are considered here.

In Section~\ref{slfvs} we introduce the {\slfvs} and its dual before providing a heuristic 
explanation for our scaling. In Section~\ref{brownian web and net} we provide a self-contained account of the necessary 
background on the 
Brownian web and net. Our main result is then stated formally in Theorem~\ref{result dimension one},
which is proved in Sections~\ref{sec:left right paths}-\ref{sec:BN conv}.
Finally, Section~\ref{numerics} presents a brief
numerical exploration of the effect of allowing the random walk paths 
of the dual process to cross on the positions at time one of the left-most
and right-most paths emanating from a point.

~

\noindent
{\bf Acknowledgement}

This work forms part of the DPhil thesis of the third author, \cite{straulino:2014}. We would like to
thank the examiners, Christina Goldschmidt and Anton Wakolbinger for their careful reading of the result and their detailed feedback.
We should also like to thank three anonymous referees for detailed and
helpful comments. Just after we submitted the original version of this paper
(and posted it on the arXiv), \cite{SSS2015} appeared, providing a 
convenient set of
criteria for convergence to the Brownian net. Following the suggestion of 
one of the referees, in this version we have adapted our proofs to exploit
those criteria.

\section{The {\slfvs} and its dual}
\label{slfvs}

In this section we introduce the set of branching and coalescing paths with which our main result is
concerned. They arise as the dual to a special instance of the {\slfvs}. The reader familiar with
the {\slfvs} can safely refer to Definition~\ref{dualprocessdefn} (and the three lines preceding it) for notation, take note of Remark~\ref{impact one}, and then
skip to Section~\ref{brownian web and net}.

\subsection{The {\slfvs}}
\label{forwards in time model}

The Spatial $\Lambda$-Fleming-Viot process ({\slfv}) without selection
was introduced in \cite{E2008, BEV2010}. 
In fact the name does not refer to a single process, but rather to a framework for modelling the 
dynamics of frequencies of different genetic types found within 
a population that is evolving in a spatial continuum. 
It is distinguished from the classical models of population biology in that reproduction is based on 
`events' rather than individuals. This
introduces density dependence into reproduction in such a way that the 
clumping and extinction which plagues classical models is overcome, whilst the model 
remains analytically tractable. For a survey of the {\slfv} we refer to \cite{BEV2013}.

There are very many different ways in which to introduce selection into the {\slfv}. Here we adapt the
approach typically adopted to introduce selection into the Moran model 
of classical population genetics. A full motivation of this approach can be found in \cite{EVY2014}, to which we refer the reader.

We suppose that the population is divided into two genetic types, which we denote $a$ and $A$, and is evolving in a geographical space which is modelled by $\R$. 
It will be convenient to index time by the whole of $\R$. At each time $t$, the population  
will be represented by a random function $\{w_t(x),\, x\in \R\}$ 
defined, up to a Lebesgue null set of $\R$, by
$$
w_t(x):= \hbox{ proportion of type }a\hbox{ at spatial position }x\hbox{ at time }t.
$$
A construction of an appropriate state space for $x\mapsto w_t(x)$ can be found in \cite{VW2013}. 
Using the identification
$$
\int_{\R} \big\{w(x)f(x,a)+ (1-w(x))f(x,A)\big\}\, dx=
\int_{\R\times \{a,A\}} f(x,\kappa) {\cal M}(dx,d\kappa), 
$$
this state space is in one-to-one correspondence with the space
${\cal M}_\lambda$ of measures on $\R\times\{a,A\}$ with `spatial marginal' Lebesgue measure,
which we endow with the topology of vague convergence. By a slight abuse of notation, we also denote the
state space of the process $(w_t)_{t\in\R}$ by 
${\cal M}_\lambda$.

\begin{defn}[One-dimensional {\slfv} with selection ({\slfvs})]
\label{slfvdefn}
Fix $\mc{R}\in(0,\infty)$ and $\upsilon\in (0,1]$ and let $\mu$ be a finite measure on $(0,\mc{R}]$.
Further, let $\Pi$ be a Poisson point process on 
$\R\times \R\times (0,\infty)$ with intensity measure 
\begin{equation}\label{slfvdrive}
dx\otimes dt\otimes \mu(dr).
\end{equation}
The one-dimensional {\em spatial $\Lambda$-Fleming-Viot process with selection} ({\slfvs}) 
driven by~\eqref{slfvdrive} is the ${\cal M}_\lambda$-valued process $(w_t)_{t\in\R}$ with dynamics given as follows.

If $(x,t,r)\in \Pi$, a reproduction event occurs at time $t$ within the closed interval $[x-r,x+r]$. With probability $1-\v{s}$ the event $(x,t,r)$ is {\em neutral}, in which case:
\begin{enumerate}
\item Choose a parental location $z$ uniformly at random within 
$(x-r,x+r)$, 
and a parental type, $\kappa$, according to $w_{t-}(z)$, that is
$\kappa=a$ with probability $w_{t-}(z)$ 
and $\kappa=A$ with probability $1-w_{t-}(z)$.
\item For every $y\in [x-r,x+r]$, set 
$w_t(y) = (1-\upsilon)w_{t-}(y) + \upsilon\ind_{\{\kappa=a\}}$.
\end{enumerate}
With the complementary probability $\v{s}$, $(x,t,r)$ corresponds to a {\em selective} event
within $[x-r,x+r]$ at time $t$, in which case:
\begin{enumerate}
\item Choose two distinct `potential' parental locations $z,z'$ independently 
and uniformly at random within $(x-r,x+r)$, and at each of these locations 
`potential' parental types 
$\kappa$, $\kappa'$, according to $w_{t-}(z), w_{t-}(z')$ respectively.
\item For every $y\in [x-r,x+r]$ set 
$w_t(y) = (1-\upsilon)w_{t-}(y) + \upsilon\ind_{\{\kappa =\kappa'=a\}}$.
Declare the parental location to be $z$ if $\kappa=\kappa'=a$ or
$\kappa=\kappa'=A$ and to be $z$ (resp. $z'$) if $\kappa=A,\kappa'=a$
(resp. $\kappa=a, \kappa'=A$).
\end{enumerate}
\end{defn}
In fact this is a very special case of the {\slfvs} introduced in 
\cite{EVY2014}, and even more special than those constructed in 
\cite{EK2014}, but it already provides a rich class of models. We use
the assumption that $\mu$ has bounded support in 
Section~\ref{sec:left right paths},
but it is far from necessary for the construction of the process.
The assumption that parental locations are sampled uniformly from 
$(-r,r)$ has become standard in the literature, but at no point do
we use it; 
our proofs work equally well for any symmetric distribution on $(-r,r)$.
The parameter $\upsilon$ is often refered to as the `impact' of the event. 
It can be loosely thought of as inversely proportional to the local 
population density. For our rigorous results we shall take $\upsilon=1$, 
meaning that during a reproduction event, {\em all} individuals in the 
affected region are replaced.

\subsection{The dual process of branching and coalescing lineages}
\label{dual process}

Our primary concern in this paper is the dual process of the {\slfvs}, a system of branching and coalescing paths 
that encodes all the \textit{potential ancestors} of individuals in a sample from the population. 

When there is no selection, the dual contains only coalescing random walks, and each such walk corresponds to the ancestral lineage 
$\ell$ of some individual; meaning that $\ell$ traces out the locations in space-time occupied
by the ancestors of that individual. 

If selection is present, then, 
at a selective event, we cannot determine the genetic types of the potential ancestors of the event
(and hence the type and location of the actual ancestor)
without looking further into the past.
To avoid intractable non-Markovian dynamics, in this case we define a dual which traces all the locations in
space-time which {\em could} have contained ancestors of a sample $S$ from the contemporary population.
This leads to a system of branching and coalescing random walks, tracing all 
the \textit{potential} ancestral lineages.

The dynamics of the dual are driven by the same Poisson point process of events, $\Pi$, that drove the forwards in 
time process. The distribution of this Poisson point process is invariant under time reversal and so we shall 
abuse notation by reversing the direction of time when discussing the dual.

We suppose that at time $0$ (which we think of as `the present') we sample $k$ individuals from locations 
$x_1,\ldots ,x_k$, and we write $\xi_t^1,\ldots ,\xi_t^{N_t}$, for the locations of the $N_t$ 
\textit{potential ancestral lineages} that make up our  dual at time $t$ before the present. 

\begin{defn}[Branching and coalescing dual]
\label{dualprocessdefn}
Fix ${\mathcal R}\in (0,\infty)$.
Let $\Pi$ be a Poisson point process on $\IR\times\IR\times (0,\infty)$ with intensity measure
$$dx\otimes dt\otimes \mu(dr)$$
where $\mu$ is a finite measure on $(0,{\mathcal R}]$.
The branching and coalescing dual process $(\Xi_t)_{t\geq 0}$ is the
$\bigcup_{n\geq 1}\R^n$-valued Markov process with dynamics defined as follows:  
At each event $(x,t,r)\in \Pi$, with probability $1-\v{s}$, the event is neutral:
\begin{enumerate}
\item for each $i$ such that $\xi_{t-}^i\in [x-r,x+r]$, mark 
the $i^{th}$ lineage with probability $\upsilon$, independently over $i$ and of 
the past;
\item if at least one lineage is marked, 
all marked lineages disappear and are replaced by a single
lineage (the `parent' of the event), whose location at time $t$
is drawn uniformly at random 
from within $(x-r,x+r)$.
\end{enumerate}
With the complementary probability $\v{s}$, the event is selective:
\begin{enumerate}
\item for each $i$ such that $\xi_{t-}^i\in [x-r,x+r]$, mark 
the $i^{th}$ lineage with probability $\upsilon$, independently over $i$ and of 
the past;
\item if at least one lineage is marked,
all marked lineages disappear and are replaced by {\em two} lineages (the `potential parents' of the event), whose (almost surely distinct)
locations are drawn independently and uniformly from
within $(x-r,x+r)$.
\end{enumerate}
In both cases, if no lineage is marked, then nothing happens.
\end{defn}
A potential ancestral lineage is any path obtained by following the locations of potential ancestors of an individual in the sample; 
whenever the potential ancestor corresponding to the 
current position of the lineage is marked in an event, the path jumps to the location of the potential parent
of the event (if the event is neutral) or to the location 
of (either) one of the potential parents (if the event is selective).
Of course there are now many such paths corresponding to each individual in the sample. 

\begin{remark}
\label{impact one}
If we take the impact $\upsilon=1$, then paths of the dual process cannot cross one another. We will
impose this condition for our main result.
\end{remark}

This dual process is the {\slfvs} analogue of the Ancestral Selection Graph (ASG), 
introduced in the companion papers \cite{krone/neuhauser:1997} and \cite{neuhauser/krone:1997}, which describes all the potential ancestors 
of a sample from a population evolving according to the Wright-Fisher diffusion with selection. 
Duality of this type can be expressed in several different ways, but perhaps the simplest is  
the statement that the ASG is the 
moment dual of the diffusion. 
To establish the analogous duality for the {\slfvs}, we would need to be able to identify 
$\E[\prod_{i=1}^n w_t(x_i)]$ for any choice of points $x_1,\ldots ,x_n\in\R$. The difficulty 
is that the {\slfvs} $w_t(x)$ is only defined at Lebesgue almost every point 
$x$ and so we have to be satisfied with a `weak' moment duality.
\begin{prop}\label{prop: dual}[\cite{EVY2014}]
The spatial $\Lambda$-Fleming-Viot process with selection is 
dual to the process $(\Xi_t)_{t\geq 0}$ in the sense that 
for every $k\in \N$ and $\psi\in C(\R^k)\cap L^1(\R^k)$, we have
\begin{align}
\E_{w_0}\bigg[\int_{\R^k} & \psi(x_1,\ldots,x_k)\bigg\{\prod_{j=1}^k w_t(x_j)\bigg\}\, dx_1\ldots dx_k\bigg] \nonumber\\
& = \int_{\R^k} \psi(x_1,\ldots,x_k)\E_{\{x_1,\ldots,x_k\}}\bigg[\prod_{j=1}^{N_t} w_0\big( \xi_t^j\big)\bigg]\, dx_1 \ldots dx_k. \label{dual formula}
\end{align}
\end{prop}
In fact, a stronger form of this duality holds, in which the forwards in time process of allele frequencies 
and the process of potential ancestors of a sample are realised on the same probability space through a lookdown 
construction; see \cite{VW2013} for the case without selection and \cite{EK2014} for the general case.

From now on 

\centerline{\textit{forwards in time refers to forwards for the dual process},}

\noindent
i.e.~the reversal of that in
Definition~\ref{slfvdefn}.

\subsection{The scaling}
\label{scaling}

We shall keep the impact of each reproduction event 
(i.e.~the parameter $\upsilon$) fixed, 
but we rescale the strength $\v{s}$ of selection. 
In addition we perform a diffusive rescaling of time and space. 
For our main result we require $\upsilon=1$, but the heuristic argument
presented here and the numerical experiments of Section~\ref{numerics},
suggest that there should be a non-trivial limit for any fixed
$\upsilon\in (0,1)$.
Let us now describe the appropriate rescaling. The stages of our rescaling are indexed by $n\in\N$.

Recall that $\mu$ is a finite measure on $(0,\mc{R}]$.
For each $n\in\N$, define the measure $\mu^n$ by
$\mu^n(A)=\mu(n^{1/2}A),$
for all Borel subsets $A$ of $\R$. 
At the $n$th stage of the rescaling, our rescaled dual is driven by
the Poisson point process $\Pi^n$ on $\R\times \R\times [0,\infty)$ 
with intensity
\begin{equation}\label{rescalingeq}
n^{1/2}\,dx \otimes n\,dt \otimes \mu^n(dr).
\end{equation}
The $\sqrt{n}$ in front of $dx$ arises since the rate at which centres of events
fall in an interval of length $l$ in the rescaled process is the 
rate at which they fall in an interval of length $\sqrt{n}l$ in the
unscaled process.
Each event of $\Pi^n$, independently, is neutral with probability $1-\v{s}_n$
and selective with probability $\v{s}_n$, where 
$\v{s}_n=\alpha/\sqrt{n}$ for some $\alpha\in (0,\infty)$. Thus, the $n^{th}$ rescaling of our dual process is precisely Definition~\ref{dualprocessdefn} with~\eqref{rescalingeq} in place of~\eqref{slfvdrive}.

Although not obvious for the {\slfvs} itself, when considering the dual process it is not hard to understand why the scaling above should lead to a non-trivial limit. If we ignore the selective events, 
then each ancestral lineage follows a compound Poisson process and rescales to a (linear time change of) 
Brownian motion. Now, consider what happens at a selective event. The two new lineages
are born at a separation of order $1/\sqrt{n}$. If we are to `see' both lineages in the limit then they must move apart 
to a separation of order $1$ (after which they might, possibly, coalesce 
back together). 
Ignoring possible interactions with other lineages, the probability that a pair of lineages achieves such a separation
is of order $1/\sqrt{n}$. 
Therefore, in order to obtain a non-trivial limit (which 
differs from that in the absence of selection) we 
need ${\mathcal O}(\sqrt{n})$ such branches per scaled unit of time, 
so we take $n\v{s}_n=\alpha \sqrt{n}$ or $\v{s}_n=\alpha/\sqrt{n}$. 
(This argument can also be used to 
identify the correct scaling of $\v{s}_n$ in order to obtain a non-trivial limit in higher dimensions, see
\cite{EFP2016}.)

Evidently we can extend the duality of Proposition~\ref{prop: dual} 
to lineages that are sampled at different 
times. 
For each point $p=(x,t)\in\R^2$, we think of an individual living at $(x,t)$ and, at the $n$th
stage of the rescaling, construct the set $\mc{P}^\uparrow_n(p)$ 
of the potential ancestral lineages of the individual at $p$. (The reason for the uparrow in 
the notation will become clear in Section~\ref{pathsandarrows2}.)
Thus $P^\uparrow_n(p)$ is a set of branching and coalescing
paths. 
Our main result will concern the limit when we 
consider the union of such sets of paths as $p$ ranges over a countable dense set of
space-time points. 

\section{The Brownian net}
\label{brownian web and net}

In order to state a precise result, we must introduce the Brownian net and, in particular, 
the state space in which convergence takes place.
A short introduction to the Brownian web and net is provided in this section. For a detailed survey of the surrounding literature, see \cite{SSS2015}.

Once again, the reader familiar with this area can note our modification 
of the usual state space (detailed in Section~\ref{intro statespace}) and Remark~\ref{constants} for
terminology, and then skip to the statement of our main result, which can be found in Section~\ref{intro d=1}. 

\subsection{The state space}
\label{intro statespace}

We now introduce the state space for our processes. Since our branching and 
coalescing paths are only c\`adl\`ag (not continuous), 
to capture the convergence of $\mc{P}^\uparrow_n(p)$ we will need a 
modification of the state space (introduced by \cite{FIN2004}) that is commonly used for the Brownian web and net.

For $s\in[-\infty,\infty]$, we set
\begin{equation*}
D[s]=\big\{f:[s,\infty]\to [-\infty,\infty]\-f\text{ is {\cadlag} on }[s,\infty]\cap(-\infty,\infty)\big\}.
\end{equation*}
For $f\in D[s]$, it will be convenient to define $\sigma(f)=\sigma_f=s$ to be the first time at which $f$ is defined. We set
\begin{equation}\label{Mdef}
M=\bigcup_{t\in[-\infty,\infty]}D[t].
\end{equation}

For each $s\in[-\infty,\infty]$ and $f\in D[s]$ we define a function $\bar{f}$ as follows. Let 
$\kappa_t=\text{tanh}^{-1}(t)$
and note that $\kappa$ is an order preserving homeomorphism between 
$[-1,1]$ and $[-\infty,\infty]$. (The specific choice of the function $\tanh$ is
a convention in the literature. We use the symbol $\kappa$ in place of $\tanh$ to denote a change of time rather than rescaling of space.)
Then if $f\in M$ we define
\begin{equation}\label{fbar0}
\bar{f}(t)=\frac{\tanh(f(\kappa_t))}{1+|\kappa_t|} 
\end{equation}
for $t\in[\kappa^{-1}(\sigma_f),1]$. It follows immediately that $\bar{f}$ is {\cadlag}.

In Section~\ref{Gsec} we define a generalization $\rho$ of the Skorohod 
metric that acts on {\cadlag} paths with possibly different starting times. 
We show (in Section~\ref{useG}) that
\begin{equation}\label{dMdef}
d_M(f_1,f_2)=\rho(\bar{f}_1,\bar{f}_2)\vee|\tanh(\sigma_{f_1})-\tanh(\sigma_{f_2})|
\end{equation}
is a pseudo-metric on $M$. 
In standard fashion, from now on we implicitly work with equivalence classes of $M$ and, with mild abuse of notation, treat $(M,d_M)$ as a metric space.
In view of~\eqref{fbar0}, the intuition for~\eqref{dMdef} is that convergence in $(M,d_M)$ can be described as 
local Skorohod convergence of the paths
plus convergence of the starting times. 

If we restrict to continuous paths and replace $\rho$ with 
the usual $L^\infty$ distance, then we recover the space $(\wt{M},d_{\wt{M}})$ 
introduced by \cite{FIN2004}, see~\eqref{FINRspace}. Convergence in the corresponding metric on continuous paths can be described as locally uniform convergence of paths
plus convergence of starting times. 

We define the set $\mc{K}(M)$ of compact subsets of $M$, equipped with the Hausdorff 
metric, $m$, and including the empty set $\emptyset$ as an isolated point. We show 
(in Section \ref{useG}) that $(M,d_M)$ is complete and separable; the space $\mc{K}(M)$ inherits 
these properties. 
Similarly, we write ${\cal K}(\wt{M})$ for the space of all compact subsets of $\wt{M}$.

\begin{figure}[]
    \centering
    \includegraphics[width=8cm]{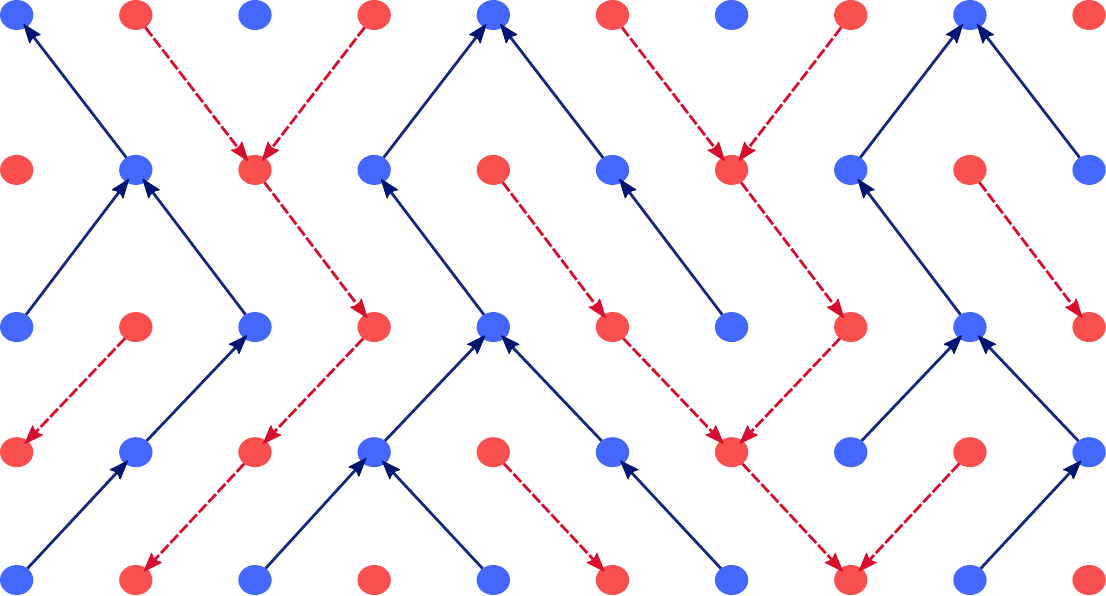}
    \caption{\textbf{Self duality of systems of coalescing random walks on the diamond lattice 
that converge to the Brownian web. }
The blue arrows represent the forwards in time coalescing random walks, while the red arrows represent 
the backwards in time dual.}\label{dualweb}
\end{figure} 

\subsection{The Brownian web and net}
\label{bw bn sec}

\cite{arratia:1979} was the first to observe that the Brownian web exhibits a self-duality. 
It is most easily understood by first considering the prelimiting system of coalescing simple random
walks, one started from each point of the diagonal space-time lattice. 
As illustrated in Figure~\ref{dualweb},
one can think of each path in the prelimiting system as the concatenation of a series of arrows,
representing the jump made by the path out of each point of $\Z$ at each time $t\in\Z$, and
there is then a natural dual system
of arrows (on the dual lattice), pointing in the opposite direction of time, which `fills out the gaps' between the 
walkers forwards in time. It is not hard to convince oneself that the law of the resultant system
of backwards paths is equal to that of the forwards system, rotated by 
$180$ degrees about the origin $(0,0)$.
Under diffusive rescaling, the forwards and backwards 
systems converge jointly to a pair $({\cal W}, \widehat{\cal W})$, known as the double Brownian web,
in which ${\cal W}$ is the Brownian web and the dual web $\widehat{\cal W}$ has the same law as ${\cal W}$
rotated by $180$ degrees.

\cite{SS2008} showed how to obtain an analogue of the Brownian web, 
which they dubbed the {\em Brownian net}, as the scaling limit of the paths traced out by
a system of branching and coalescing simple random walks. 
If there is a branch at $(x,t)$, then the random walker at point $x$ at time $t$ has two offspring which
it places at $(x-1)$ and $(x+1)$, so that the space-time point $(x,t)$ is connected by 
paths to each of $(x-1, t+1)$ and $(x+1, t+1)$.
In order to obtain a non-trivial limit, 
the branching probability of each path in each time step is scaled to be ${\mathcal O}(1/\sqrt{n})$, 
corresponding exactly to the scaling in the
dual to the {\slfvs}; that is at the $n$th stage of the rescaling the 
probability that two paths, one stepping left and one stepping right,
emanate from a given point is $\zeta/\sqrt{n}$. 

In contrast to the Brownian web, the Brownian net will have a multitude of paths coming out of each space-time point.
The key to its characterisation is that it has a well-defined left-most and right-most path, which we denote $l_z$ and
$r_z$ respectively, emanating from each point $z=(x,t)\in\R^2$ and these determine what is called a left-right Brownian web. Essentially, the set of left-most (resp.~right-most) paths form a Brownian web with a leftwards (resp.~rightwards) drift. Thus, for any deterministic pair of $k$-tuples of points 
$(z_1,\ldots, z_k)$, $(z_1',\ldots ,z_{k'}')$, the left-most paths 
$l_{z_1},\ldots ,l_{z_k}$ are distributed as coalescing 
Brownian motions with drift $\zeta$ to the left, and 
the right-most paths $r_{z_1'},\ldots r_{z_{k'}'}$ are distributed as
coalescing Brownian motions with drift $\zeta$ to the right. 

Before we can fully describe the Brownian net, we must explain how a left-most path $l_z=l_{(x,s)}$ and a right-most path $r_{z'}=r_{(x',s')}$
interact. Their joint evolution after time $s\vee s'$ is the unique weak solution to the left-right stochastic differential
equation
\begin{equation}
\label{sslrsde}
\begin{split}
dL_t &=\xi {\mathbf 1}_{\{L_t\neq R_t\}}dB_t^l+\xi {\mathbf 1}_{\{L_t=R_t\}}dB_t^c-\zeta dt,\\
dR_t &=\xi {\mathbf 1}_{\{L_t\neq R_t\}}dB_t^r+\xi {\mathbf 1}_{\{L_t=R_t\}}dB_t^c+\zeta dt,\\
\end{split}
\end{equation}
where $B_t^l$, $B_t^r$ and $B_t^c$ are independent standard Brownian motions and if $s<t$ then $L_s\leq R_s\ra L_t\leq R_t$.
\cite{SS2008} proved (weak) existence and uniqueness of the solution to this system. 

A straightforward extension of~\eqref{sslrsde} is sufficient to specify the joint distribution of any finite collection of left-right paths, which are known as left-right coalescing Brownian motions.

\begin{remark}
\label{constants}
In \cite{SS2008} the drift parameter $\zeta$ of the left-right stochastic differential equation 
used to construct the Brownian net is allowed to vary but the diffusion constant, 
$\xi^2$, of the Brownian motions is always taken to be one. 
Applying a linear time change to their construction yields general $\xi^2$ and we will 
use such webs and nets (and results from elsewhere extended trivially to apply to them) 
without further comment. We shall refer to the Brownian net corresponding to the left-right system~\eqref{sslrsde}
as the net with drift $\zeta$ and diffusion constant $\xi^2$.
\end{remark}

It remains to give a rigorous characterization of the Brownian net. One last ingredient is required.

\begin{defn}\label{crossingdef}
Let $\alpha:[\sigma_\alpha,\infty)\to\R$ and 
$\alpha':[\sigma_{\alpha'},\infty)\to\R$ be paths. We say $\alpha$ crosses 
$\alpha'$ from left to right at time $t\in\R$ if there exists 
$t^-<t$ and $t^+>t$ such that
$\alpha(t^-)-\alpha'(t^-)<0$ and $\alpha(t^+)-\alpha'(t^+)>0$ and 
$t=\inf\{s\in(t^-,t^+)\-(\alpha(t^-)-\alpha'(t^-))(\alpha(s)-\alpha'(s))<0\}$.

We define a crossing of $\alpha'$ by $\alpha$ from right to left analogously. We say that $\alpha$ crosses $\alpha'$ if it does so from either left to right or right to left.
\end{defn}

Given two paths $\alpha$ and $\alpha'$ which cross at say, time $t$, we can define a new path $g$ by following $\alpha$ up until time $t$, and subsequently following $\alpha'$. This procedure is known as \emph{hopping} from $\alpha$ to $\alpha'$ at the (crossing) time $t$. Given a set $P$ of paths, $\mc{H}_{cross}(P)$ is defined to be the set of paths obtained by hopping a finite number of times between paths within $P$.

\begin{defn}[\cite{SS2008}]
The Brownian net $\mc{N}$ is the $\mc{K}(\wt{M})$ valued random variable whose distribution is uniquely determined by the following properties:
\begin{enumerate}
\item For each deterministic $z\in\R^2$, almost surely $\mc{N}$ contains a unique left-most path $l_z$ and a unique right-most path $r_z$.
\item For any finite deterministic set of points $z_1,\ldots,z_k,z'_1,\ldots,z'_{k'}\in\R^2$, the collection of paths $l_{z_1},\ldots,l_{z_k},r_{z'_1},\ldots,r_{z'_{k'}}$ has the distribution of a family of left-right coalescing Brownian motions.
\item For any deterministic dense countable sets $\mc{D}^l,\mc{D}^r\sw\R^2$,
$$\mc{N}=\overline{\mc{H}_{cross}(\{l_z\-z\in\mc{D}^l\}\cup\{r_z\-z\in\mc{D}^r\})}.$$
\end{enumerate}
\end{defn}

The proof of our main result rests on verifying the conditions of
Theorem~\ref{thm:BN conv criteria}, which provides criteria under which
a sequence of processes converges to the Brownian net. It
is obtained by combining Theorem~6.11 and 
Remark~6.12 of \cite{SSS2015}. To state it, we require the
notion of a wedge.
Let $(\hat{X}^l_n)$ and $(\hat{X}^r_l)$ be two random sets of paths such 
that their rotations by $180$ degrees about $(0,0)$ 
are $\mc{K}(\wt{M})$ valued random variables.
Take $\hat{l}\in\hat{X}^l_n$ and $\hat{r}\in\hat{X}^r_n$,
defined on time intervals $(-\infty, \sigma(\hat{l})]$ and 
$(-\infty, \sigma(\hat{r})]$ respectively. We write 
$s=\sigma(\hat{l})\wedge \sigma(\hat{r})$ for the largest time at which both
paths are defined. Suppose that $\hat{r}(s)<\hat{l}(s)$ and define
$T:=\sup\{t<s :\hat{r}(t)=\hat{l}(t)\}$ to be the first time time the 
paths meet (as we trace backwards in time).
We call the open set 
\begin{equation}
\label{wedge defn}
W(\hat{r},\hat{l}):=\{(x,u)\in\R^2: T<u<s,\hat{r}(u)<x<\hat{l}(u)\}
\end{equation}
a wedge. This set is illustrated in Figure~\ref{wedge}. 
We say that a path $\pi$ started at time $\sigma_{\pi}$ enters W 
from the outside if there exists $\sigma_{\pi}\leq u<t$ 
such that $(\pi(u),u)\notin \overline{W}$ and $(\pi(t),t)\in W$. Here, $\overline{W}$ denotes the closure of $W$.
\begin{figure}[]
    \centering
    \includegraphics[width=5cm]{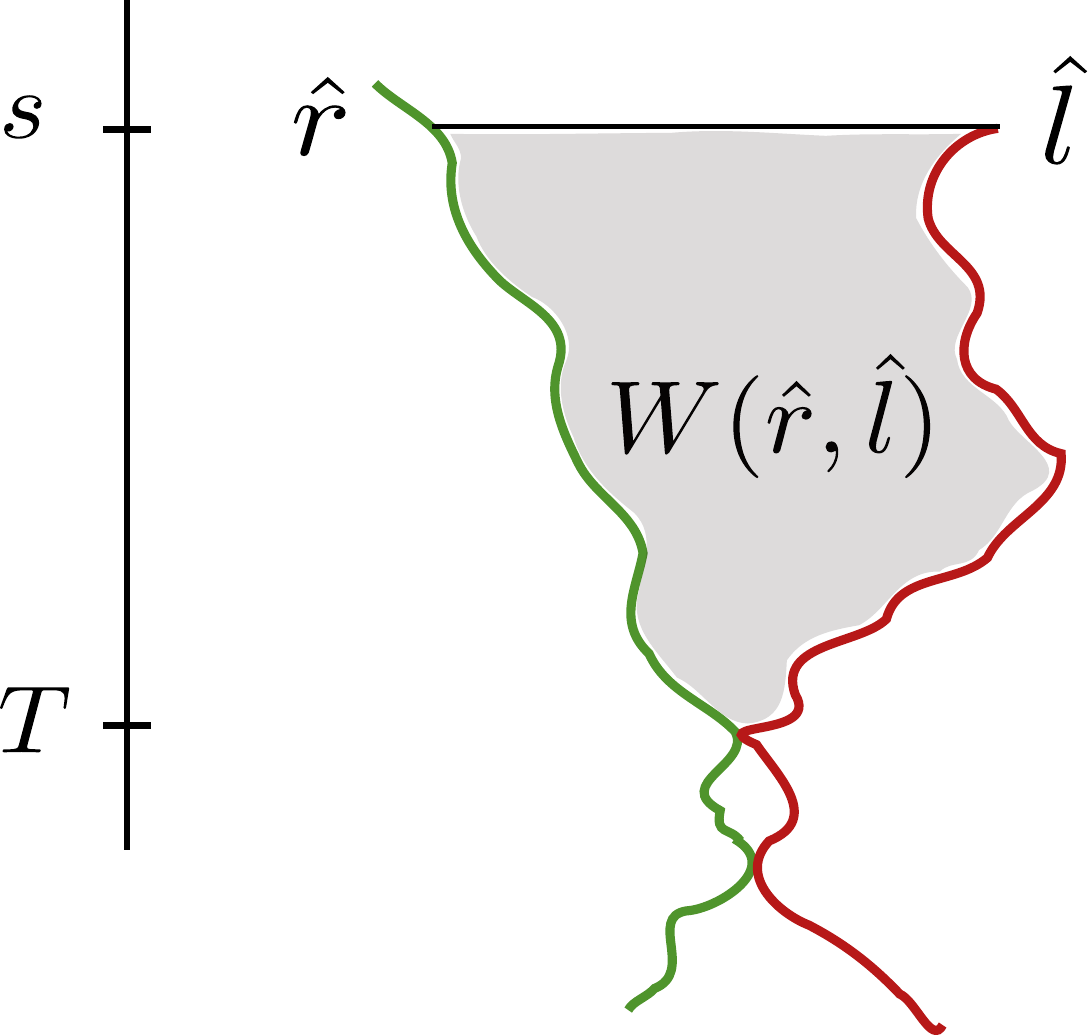}
    \caption{\textbf{A wedge. } In the notation of~\eqref{wedge defn},
the path $\hat{r}$ is in $\hat{X}_n^r$, 
and the path $\hat{l}$ is in $\hat{X}_n^l$. The last time at
which both paths are defined is $s$, in this case given by $\sigma(\hat{l})$; $\hat{r}(s)<\hat{l}(s)$ and, tracing 
backwards in time, $T$ is the first time at which $\hat{r}=\hat{l}$. The wedge is the 
shaded region. }\label{wedge}
\end{figure}
\begin{theorem}[\cite{SSS2015}]\label{thm:BN conv criteria}
Let $(X_n^l)$ and $(X_n^r)$ be two sequences of $\mc{K}(\wt{M})$ valued 
random variables. 
Let $(\hat{X}^l_n)$ and $(\hat{X}^r_l)$ be two random sets of paths such 
that their rotations by $180$ degrees about $(0,0)$ 
are $\mc{K}(\wt{M})$ valued random variables. Set
$X_n=\mc{H}_{cross}(X_n^l\cup X_n^r)$ and $\hat{X}_n=\mc{H}_{cross}(\hat{X}_n^l\cup \hat{X}_n^r)$.

Suppose that:
\begin{itemize}
\item[$(\mathscr{A})$] Paths in $X_n^l$ (resp.~$X_n^r$) do not cross. No path in $X_n$ crosses a path of $X_n^l$ from right to left, and no path in $X_n$ crosses a path of $X_n^r$ from left to right. No path in $X_n^l$ crosses a path of $\hat{X}_n^l$, and no path of $X_n^r$ crosses a path of $\hat{X}^r_n$. 
\item[$(\mathscr{B})$] For any $k\in\N$, and any $(z_1,\ldots,z_{2k})\sw\R\times\R$ there exists a convergent sequence $$(l_{n,1},\ldots,l_{n,k},r_{n,1},\ldots,_{n,k}),$$ where $l_{n,i}\in X_n^l, r_{n,i}\in X_n^r$, whose limit (in distribution, in $\wt{M}^{2k}$, as $n\to\infty$) is a collection of left/right coalescing Brownian motions started at $(z_1,\ldots,z_{2k})$.
\item[$(\mathscr{C})$] Whenever $k\in\N$ and $\hat{l}_n\in \hat{X}_n^l$ and $\hat{r}_n\in\hat{X}^r_n$ are such that $(\hat{l}_n,\hat{r}_n)$ converges (in $\wt{M}^2$, in distribution, as $n\to\infty$) to left/right Brownian motions $(\hat{l},\hat{r})$, the first meeting time of $\hat{l}_n$ with $\hat{r}_n$ also converges in distribution to the first meeting time of $\hat{l}$ with $\hat{r}$.
\item[$(\mathscr{D})$] Paths of $X_n$ do not enter wedges of $\hat{X}_n$ from the outside.
\end{itemize}
Then, $X_n$ converges (in $\mc{K}(\wt{M})$, in distribution) to the Brownian net.
\end{theorem}

\subsection{Statement of the main result}
\label{intro d=1}

We are finally in a position to give a formal statement of our result. 
Recall from Section~\ref{scaling},
that $P^\uparrow_n(p)$ is the set of potential ancestral lineages of the individual at $p\in \IR^2$ 
at the $n$th stage of our rescaling. 

Let $(\mc{D}_n)_{n\in\N}$ be an increasing sequence of countable subsets of $\R^2$ such that, for each $n$, $\mc{D}_n$ is locally finite, and as $n\to\infty$ the set $\mc{D}_n$ becomes everywhere dense.

We define
$\mathscr{A}(\mc{D}_n)=\bigcup_{p\in\mc{D}_n} \mc{P}^\uparrow_n(p).$
The set $\mathscr{A}(\mc{D}_n)$ contains the potential ancestral lineages of all $p\in\mc{D}_n$. 
However, $\mathscr{A}(\mc{D}_n)$ is not an element of $\mc{K}(M)$, since it is not a closed subset of $M$, and 
so at the very least we should consider its closure. This requires that we augment $\mathscr{A}(\mc{D}_n)$ to also include ancestral lineages $f$ that extend backwards in time until time $-\infty$, and we define $f(-\infty)=0$ for such $f$. We include $\infty$ in the domain of 
each path $f$ by defining $f(\infty)=0$. Additionally, define 
the boundary paths
\begin{equation}\label{bdrypaths}
\mathcal{B}=\{f(\cdot)=-\infty\-\sigma_f\in[-\infty,\infty]\}\cup\{f(\cdot)=\infty\-\sigma_f\in[-\infty,\infty]\}.
\end{equation}
We then set $\mc{P}^\uparrow_n(\mc{D}_n)=\mathscr{A}(\mc{D}_n)\cup\mc{B}$. 
Lemma~\ref{compact} shows that $\mc{P}^\uparrow_n(\mc{D}_n)$ is an element of $\mc{K}(M)$.

Recall from Definition~\ref{dualprocessdefn} that 
$\upsilon$ is the probability that an ancestral lineage that lies
in $[x-r,x+r]$ at time $t-$ is affected by
the event $(x,t,r)$ and that $\v{s}_n=\alpha/\sqrt{n}$ is the 
probability (at the $n$th stage of our rescaling) that an event is selective.
Our main result is the following.

\begin{theorem}\label{result dimension one}
Let $\upsilon=1$. As $n\rightarrow\infty$, $\mc{P}_n^{\uparrow}(\mc{D}_n)$ 
converges weakly to $\mc{N}$ in $\mc{K}(M)$ where,
in the terminology of Remark~\ref{constants}, $\mc{N}$ denotes the Brownian net with drift
\begin{equation}
\label{zetadef}
\zeta=\frac{2}{3}\alpha\int_0^{\mc{R}}r^2\mu(dr),
\end{equation}
and diffusion constant
\begin{equation}
\label{xidef}
\xi^2 = \frac{4}{9}\int_0^{\mc{R}}r^3\mu(dr).
\end{equation}
\end{theorem}

The proof of Theorem~\ref{result dimension one} can be found in
Sections~\ref{sec:left right paths}-\ref{sec:BN conv}. 
It rests heavily on the theory of the Brownian web and net, in particular on Theorem~\ref{thm:BN conv criteria}.  We will now place this result in the context of existing work and
outline some of the additional difficulties that are encountered in
our setting. 

Consider, first, what would happen in the absence of selection. Our dual process reduces to a system of coalescing random walks and, as proved in
\cite{berestycki/etheridge/veber:2013},
after a diffusive rescaling one recovers a system of (instantaneously) coalescing Brownian motions.
If we set $\upsilon=1$ and take the centre of the event, rather
than a randomly chosen point, as the location of the parent, then this corresponds to the process of `trajectoires d'exploration'
of \cite{micaux:2007}, who constructs a stochastic flow of maps by considering the dual started from every space-time 
point in the plane.
If we specialise still further so that all events have radius 1 and we start `exploration paths' only from
the centre of each reproduction event then we recover a {\cadlag} version of the Poisson trees of
\cite{ferrari/landim/thorisson:2004}.  By interpolation we recover the Poisson trees themselves.
In \cite{FFW2005} (see \cite{FFW2003} for a more detailed account) 
it is shown that under a diffusive rescaling the
Poisson trees converge to the Brownian web. 

Even with the simplification $\upsilon=1$, our prelimiting process is considerably more complex than that considered in \cite{SS2008}.
When lineages
are covered by the same neutral reproduction event, they coalesce. In particular, more than two lineages can coalesce in 
a single event. At selective events, when we must trace two potential parents, we can see either just branching or,
if more than one lineage lies in the region affected by the event, a 
combination of branching and coalescence (see Figure~\ref{complicated}).

\begin{figure}[]
    \centering
    \includegraphics[width=10cm]{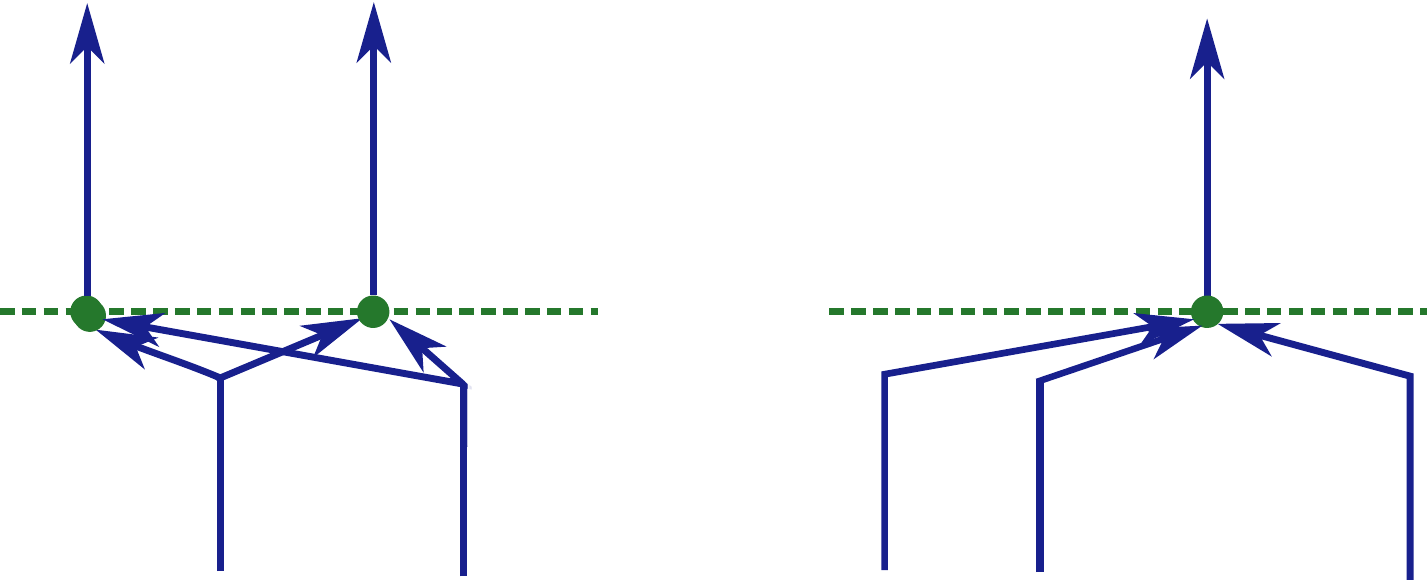}
    \caption{\textbf{Complications. }The diagram on the left illustrates the way in which
    paths can both coalesce and branch through the same event. 
    The second diagram presents a case of multiple collisions in a `neutral' event. }\label{complicated}
\end{figure}

Further complications compared to systems of branching and coalescing simple random walks arise since
(a) our ancestral lineages jump at random times and the displacement caused by such jumps is random; and
(b) the motion of distinct ancestral lineages becomes dependent when they are within distance $2\mc{R}$ of each other.

 In spite of the additional complexity, it still makes sense to talk about left-most and right-most paths  and this will be the key to our analysis.
In fact (b) can be handled through elementary arguments; it turns out that the time periods during which ancestral lineages are `nearby but not coalesced' are too brief to 
affect the limit. 

In order to overcome (a), we must identify a dual system of (backwards in time) branching and coalescing lineages. 
At first sight, it is far from obvious that such a dual exists; in contrast
to previous work, our pre-limiting systems will not be self-dual.
We will construct a dual system with the property that, in contrast to 
Figure~\ref{dualweb}, 
after rotation by $180$ degrees, although, separately, left-most and 
right-most paths in the dual have the same
distribution as their forwards counterparts, the joint distributions of 
the forwards and backwards systems differ. 
The dual, which is defined in Section~\ref{arrowsec} is illustrated
in Figure~\ref{arrows}.

\section{Convergence of left/right paths}
\label{sec:left right paths}

We now turn to the proof of our main result. Recall that we take 
$\upsilon=1$ so that if a lineage is in the interval 
covered by an event then it is necessarily affected by it.

\subsection{Paths and arrows}
\label{pathsandarrows2}

In order to discuss the self-dual systems of branching and coalescing lineages that converge to the 
Brownian net, we must be precise about what we mean by `branching-coalescing paths' and, in particular,
have a notation for keeping track of 
the direction of time. We shall follow \cite{FIN2004} in using segments of paths called arrows.
Loosely speaking, paths are formed by concatenating arrows. 
A path (or an arrow) is an $\R$-valued function whose domain is a subinterval of $\R$. 
If a path/arrow is \textit{forwards} (resp.~\textit{backwards}), then `moving along it' means moving along the 
image of the path forwards (resp.~backwards) with respect to the usual (resp.~reversed) order on the time domain. 
We shall use $\uparrow$ to denote forwards and $\downarrow$ to denote backwards paths.

For $a<b<c$, the concatenation of forwards paths $f:[a,b)\to\R$ and $g:[b,c)\to\R$ refers to the function $h:[a,c)\to\R$ which is equal to $f$ on $[a,b)$ and equal to $g$ on $[b,c)$. Concatenation of backwards paths is defined analogously.

When we are following a backwards path or arrow we interchange left and right, in the same way as left and right 
interchange if we reverse the direction in which we walk. For clarity, we reserve the terms 
\textit{north, south, east and west} for global directions associated to the plane $\R^2$ and use the terms 
\textit{right} and \textit{left} for local directions whose frame of reference depends on the direction in which 
we are travelling. 

\subsubsection{Forwards and backwards paths}
\label{arrowsec}

Recall from Section~\ref{scaling} that $\Pi^n$ denotes the Poisson point process that drives the system of
branching and coalescing paths at the $n$th stage of our rescaling.
We refer to each $(x,t,r)\in\Pi^n$ as an event affecting the set $\{t\}\times [x-r, x+r]$ or, equivalently, 
affecting each point $y\in [x-r,x+r]$ at time $t$. The east- and west-most points of this event are 
$(x+r,t)$ and $(x-r,t)$ respectively. To each 
$(y,s)\in(-\infty,\infty)\times\R$ we associate a 
unique forwards arrow (pointing due north) and a unique backwards arrow (pointing due south), defined as follows. Let
\begin{align*}
T^{\uparrow}_{y,s}&=\inf\{t\-\exists (x,t,r)\in\Pi^n, y\in [x-r,x+r], t\geq s\},\\
T^{\downarrow}_{y,s}&=\sup\{t\-\exists (x,t,r)\in\Pi^n, y\in [x-r,x+r], t\leq s\},
\end{align*}
be the times of the first event (non-strictly) north of $(y,s)$ and the first event (non-strictly) south
of $(y,s)$, respectively, that affects the point $y$.
Let $\star\in\{\uparrow,\downarrow\}$. An arrow starting at $(y,s)$ is simply a path 
$\alpha^\star_{y,s}:[s,T^\star_{y,s})\to\R$ defined to be the constant function 
$\alpha^\star_{y,s}(u)=y$.
We shall call the event $(x,t,r)\in\Pi^n$ that defines $T^\star_{y,s}$ 
the {\em finishing event} of $\alpha_{y,s}$. It must be that
$\lim_{u\uparrow T^\star_{y,s}}(\alpha(u),u)=(y,T^\star_{y,s})\in [x-r,x+r]\times\{t\}$.

For each $\star\in\{\uparrow,\downarrow\}$, we can now associate two important sets of paths to each point $(y,s)$.
Let us first consider the forwards paths. The set $\mc{P}^{\uparrow}_n(y,s)$ is best described in words; it is the set of paths 
that are obtained by following the arrow $\alpha^{\uparrow}_{y,s}$ out of $(y,s)$ and then, every time we finish an arrow, 
following a new arrow that starts from (one of) the (potential) parent(s) of the finishing event of $\alpha_{y,s}$.
In other words, the forwards paths from $(y,s)$ correspond precisely to the set of potential ancestral 
lineages of an individual who lived at the point $y$ at time $s$, that we described in Section~\ref{scaling}. We include time $\infty$ into the domain of each such forwards path, and set the location at time $\infty$ to be $0$.

The set $P^{\downarrow}_n(y,s)$ of backwards paths is also best described in words. 
It is the set of paths obtained by first 
following the arrow $\alpha^{\downarrow}_{y,s}$ out of $(y,s)$ and then, 
every time we finish an arrow $\alpha^{\downarrow}_{y',s'}$:
\begin{enumerate}
\item If the finishing event of $\alpha_{y',s'}$ is neutral with, parent at $v$, then
\begin{enumerate}
\item if $y'\leq v$, follow 
the arrow out of the west-most point of the finishing event of $\alpha_{y',s'}$,
\item if $y'>v$, follow the arrow out of the east-most point of the finishing event of $\alpha_{y',s'}$.
\end{enumerate}
\item If the finishing event of $\alpha_{y',s'}$ is selective with potential parents at $v<v'$ then
\begin{enumerate} 
\item if $y'< v$, follow the arrow out of the west-most point of the finishing event of $\alpha_{y',s'}$,
\item if $y'> v'$, follow the arrow out of the east-most point of the finishing event of $\alpha_{y',s'}$,
\item if $y'\in[v,v']$, a path can follow either one of the arrows out of the east-most/west-most points of the finishing event 
of $\alpha_{y',s'}$.
\end{enumerate}
\end{enumerate}
In analogy to forwards paths, we include time $-\infty$ into the domain of each such backwards path, and set the location at time $-\infty$ to be $0$. See Figure~\ref{arrows} for an illustration of the forwards and backwards paths.

\begin{figure}[t]
\includegraphics[height=2.5in,width=6.5in]{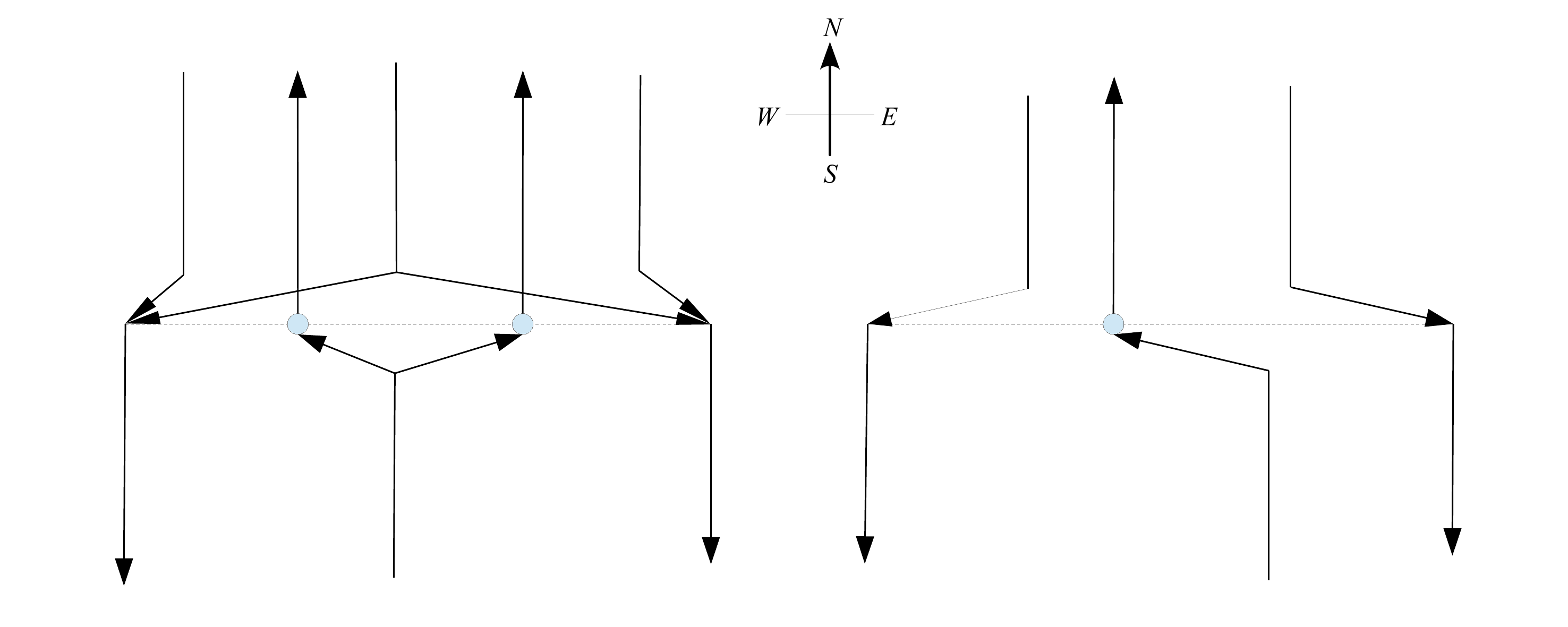} 
\caption{\label{arrows} \small The movement of forwards and backwards paths (illustrated as interpolated arrows) 
about a selective event (left) and a neutral event (right). The events are shown as finely dotted horizontal lines 
and the (potential) parent(s) as small circles. Forwards paths travel northwards and backwards paths travel southwards, 
according to the compass shown between the two diagrams.} 
\end{figure}

In keeping with our previous notation,
for each forwards/backwards path $f$, we write $\sigma(f)=\sigma_f$,
for the time at which it starts.

\subsubsection{Interpolated paths and arrows}
\label{interppathsec}

We wish to exploit the existing theory of Brownian webs and nets, which was developed in a setting restricted to 
continuous paths, and so we shall approximate the systems of ({\cadlag}) forwards and backwards paths 
of the last subsection by 
corresponding systems in which the jumps have been interpolated. 
\cite{FFW2005} achieve this by simply taking paths that interpolate between the starting points
of arrows. However, in our situation this would result in arrows which
cross each other and, worse, pass 
through reproduction events that did not previously affect them. 
Instead, we adopt a `just in time' approach to our interpolation: we find small non-overlapping intervals of 
time and space about each event in which to interpolate. 

\begin{lemma}\label{upsilonlemma}
Let $n\in\N$. For each $p=(x,t,r)\in\Pi^n$ and each $\epsilon>0$ define the set
$$\mathscr{B}_\epsilon(x,t,r)=\{(y,s)\in\R^2\-|x-y|\leq r, |t-s|\leq \epsilon\}.$$
Almost surely, there exists a map
$\Upsilon:\Pi^n\to(0,\infty)$
such that the sets $(\mathscr{B}_{\Upsilon(p)}(p))_{p\in\Pi_n}$ are distinct.
\end{lemma}

\begin{proof}
This follows essentially immediately since $\Pi^n$ has finite intensity; consequently the set of time coordinates of points of $\Pi_n$, restricted to any strip $[-K,K]\times\R\times [-\mc{R},\mc{R}]$, where $K\in\R$, has (almost surely) no limit point.
\end{proof}

Let $f^\uparrow\in\mc{P}^\uparrow(y,s)$ and let $\alpha^\uparrow_{y,s}:[s,T^\uparrow_{y,s})\to\R$ be one of the forwards 
arrows that make up $f^\uparrow$. 
Let $p=(x,t,r)$ denote the finishing event of $\alpha^\uparrow_{y,s}$
(so that, in particular, $T^\star_{y,s}=t$ a.s.).
Suppose that $\alpha'$ is the next arrow in $f^\uparrow$ and write
$z=\alpha'(T^\uparrow_{y,s})$ for its starting point.
We say that $\wt{\alpha}^{\uparrow}_{y,s}:[s,t)\to\R$ is the interpolated arrow 
of $\alpha^\uparrow_{y,s}$ if both
\begin{enumerate}
\item $\wt{\alpha}^{\uparrow}_{y,s}(u)=\alpha^\uparrow_{y,s}(u)$ for all $u\leq T^\uparrow_{y,s}-\Upsilon(p)$, and
\item $\wt{\alpha}^{\uparrow}_{y,s}(u)$ is linear on $[T^\uparrow_{y,s}-\Upsilon(p),t)$ and $\lim_{u\uparrow t}\wt{\alpha}^\uparrow_{y,s}(u)=z$.
\end{enumerate}
Note that the interpolation of an arrow depends on the path $f$ in which it is contained.

Given a forwards or backwards path $f\in\mc{P}^\uparrow_n(y,s)$, we define the continuous path
$\tilde{f}$ to be the concatenation of the 
interpolations of the arrows within $f$, and additionally setting $f(\infty)=0$ for forwards paths and $f(-\infty)=0$ for backwards paths. We define 
$$\wt{\mc{P}}^\uparrow_n(y,s)=\{\wt{f}\-f\in\mc{P}^\uparrow_n(y,s)\}$$
and define the set of interpolated backwards paths $\wt{\mc{P}}^\downarrow_n(y,s)$ in analogously.
Of course, interpolated paths are close to their equivalent non-interpolated paths.

\begin{lemma}\label{interppaths}
Let $(y,s)\in\R^2$ and let $f\in\mc{P}_n^\uparrow(y,s)$. Then 
$\sup_{t\in(\sigma_f,\infty)}|f(t)-\wt{f}(t)|<2\mc{R}n^{-1/2}$.
The analogous estimate holds for backwards paths.
\end{lemma}
\begin{proof}
Note that $\sigma_f=\sigma_{\wt{f}}$. By definition, 
in the notation of Lemma~\ref{upsilonlemma},
$f(u)=\wt{f}(u)$ unless $u$ is such that $(f(u),u)\in\mathscr{B}_{\Upsilon(x,t,r)}(x,t,r)$ 
for some $(x,t,r)\in\Pi^n$. When $(f(u),u)\in\mathscr{B}_{\Upsilon(x,t,r)}(x,t,r)$ we have $|f(u)-\wt{f}(u)|\leq 2r$. 
Since, by definition of $\Pi^n$ we have $r\leq \mc{R}n^{-1/2}$, this completes the proof.
\end{proof}

\subsubsection{Left-most and right-most paths}
\label{pathdefs}

We now associate to each $(y,s)$ four special paths. 
\begin{defn}
Left-most and right-most forward and backward paths are defined as follows.
\begin{enumerate}
\item The \textit{left-most forward} path from $(y,s)$ is the element of $\mc{P}_n^{\uparrow}(y,s)$ obtained by choosing the (forwards) arrow with the west-most potential parent, whenever a choice is available.
\item The \textit{right-most forward} path from $(y,s)$ is the element of $\mc{P}_n^{\uparrow}(y,s)$ obtained by choosing the (forwards) arrow with the east-most potential parent, whenever a choice is available.
\item The \textit{left-most backward} path from $(y,s)$ is the element of $\mc{P}_n^{\downarrow}(y,s)$ obtained by 
choosing the (backwards) arrow from the east-most point of the finishing event whenever a choice is available.
\item The \textit{right-most backward} path from $(y,s)$ is the element of $\mc{P}_n^{\downarrow}(y,s)$ obtained by 
choosing the (backwards) arrow from the west-most point of the finishing event, whenever a choice is available.
\end{enumerate}
\end{defn}

We will sometimes shorten `left-most' and `right-most' to $l$-most and $r$-most. 

For $D\sw \R^2$, $\dagger\in\{\uparrow,\downarrow\}$ and $\star\in\{l,r\}$ we define 
$$
\mc{Q}^{\star,\dagger}_n(D)=\{f\-f=f^{\star}_{y,s}\text{ is the }\dagger\text{-most path of some }(y,s)\in D\}.
$$	
Recall from Section~\ref{intro d=1} that, in order to exploit the 
compactness properties of our state space, we must also include 
some extra paths, corresponding to ancestral lineages that extend backwards
in time until $-\infty$. First, we say that a path $f:\R\to\R$ is an 
infinite extender of $\mc{Q}^{\uparrow,\dagger}_n(D)$ if there exists a 
sequence $(f_m)_{m=1}^\infty\sw\mc{Q}^{\uparrow,\dagger}_n(D)$ and a 
sequence $(t_m)$ such that $t_m\downarrow -\infty$ and $f(t)=f_m(t)$ for 
all $m$ and $t\geq t_m$. We make the corresponding definition for 
$\mc{Q}^{\downarrow,\dagger}_n(D)$ and, for $\star\in\{\uparrow,\downarrow\}$ 
and $\dagger\in\{l,r\}$ we define $\mc{Q}^{\star,\dagger,inf}_n(D)$ to be 
the set of infinite extenders of $\mc{Q}^{\star,\dagger}_n(D)$. Recall 
also the boundary paths $\mathcal{B}$ defined in~\eqref{bdrypaths}. Then, 
define
\begin{equation}\label{Paugment}
\mc{P}^{\star,\dagger}_n(D)=\mc{Q}^{\star,\dagger}_n(D)\cup \mc{Q}^{\star,\dagger,inf}_n(D) \cup {\mathcal B},
\end{equation}
and $\wt{\mc{P}}^{\star,\dagger}_n(D)=\{\wt{f}\-f\in\mc{P}^{\star,\dagger}(D)\}$ to be the corresponding sets of interpolated paths.

We now verify Condition $(\mathscr{A})$ of Theorem~\ref{thm:BN conv criteria}.
Recall, from Definition~\ref{crossingdef} what it means for two paths to cross.
\begin{lemma}\label{noncrossing}
Let $D$ be any subset of $\R^2$ and let $n\in\N$. Let $\dagger\in\{l,r\}$ and $\star\in\{\uparrow,\downarrow\}$. Then, almost surely:
\begin{enumerate}
\item For all $f^\uparrow\in\mc{P}^{\uparrow,\dagger}_n(D)$ and $f^\downarrow\in\mc{P}^{\downarrow,\dagger}_n(D)$, the paths $f^\uparrow$ and $f^\downarrow$ do not cross.
\item For all $f^\star,g^\star\in\mc{P}^{\star,\dagger}(D)$, the paths $f^\star$ and $g^\star$ do not cross.
\end{enumerate}
Further, the same results hold for interpolated paths $f^\uparrow\in\wt{\mc{P}}^{\uparrow,\dagger}_n(D)$ and $f^\downarrow\in\wt{\mc{P}}^{\downarrow,\dagger}_n(D)$.
\end{lemma}
\begin{proof}
In the first case, note that two forwards paths can only cross if they are first coalesced and are then subsequently affected by the same selective reproduction event. In the second case, note that a forward path can only cross a backwards path if both are affected by a common event. In both cases, the fact that crossing cannot occur is then an easy consequence of the definitions (or see Figure~\ref{arrows}).
\end{proof}

\begin{remark}
A forwards left-most path can cross a backwards right-most path, and a forwards right-most path can cross a backwards left-most path. Similarly, a forwards left-most path can cross a forwards right-most path (if they are both affected by the same selective event), and a backwards right-most path can cross a backwards left-most path.
\end{remark} 

Although not immediately obvious from the definition, 
the next lemma is a helpful feature of our construction.
\begin{lemma}
\label{samelaw}
A forwards left- (resp.~right-) most path has the same distribution as a backwards left- (resp.~right-) most path which has been rotated by $180$ degrees.
\end{lemma}
\begin{proof}
The proof is based on the movements of paths affected by reproduction events, which is depicted in Figure~\ref{arrows}. It suffices to consider the case of left-most paths; the case of right-most paths then follows by symmetry. 

First observe that the rate at which an event falls on (an arrow in) a path has the same distribution whether we look forwards or backwards in time and, when an event falls on a (forwards or backwards) path, the spatial position of the path will be uniformly distributed over the region affected by the event. Let us denote that position by $V$. Thus if the event corresponds to $p=(x,t,r)$, then $V$ is uniformly distributed on $[-r,r]$. 

Consider a left-most forwards path affected by a neutral event. The path jumps to the position of the parent, which we denote by $U$.
Thus, on the event $V<U$ our path jumps a distance $U-V$ to the left, and on the event $U>V$ it jumps a distance $V-U$ to the right. 
 
Now consider the left-most backwards path. Retaining the notation above, at a neutral event, 
on the event $V<U$ the path jumps to the west-most endpoint, which, once rotated by 180 degrees becomes a 
jump to the right of size $V-(-r)$. 
On the other hand, on the event $V>U$, the path jumps to the east-most endpoint, which upon rotation becomes a 
leftwards jump of magnitude $r-V$. 

Conditional on $V<U$, $V$ is uniform on $(-r,U)$, so $U-V\stackrel{d}{=}V-(-r)$. Similarly, conditional on $V>U$, $V$ is uniform on $(U,r)$ and $V-U\stackrel{d}{=}r-V$. 
Therefore, if we restrict to only neutral events, forwards left-most paths and backwards left-most paths rotated by 180 degrees have the same distribution.

Next, consider a selective event. We use a similar argument. The two potential parents are sampled uniformly from the event. We denote their positions by $U_1<U_2$. Combined with $V$, we now have three independent uniformly distributed random variables on $[-r,r]$. Let us write them in ascending order as $U^{(1)}$, $U^{(2)}$, $U^{(3)}$. The following events may occur:
\begin{enumerate}[(a)]
\item $V=U^{(1)}$, in which case $U_1=U^{(2)}$, so the path makes a rightwards jump of magnitude $U_1-V$;
\item $V\neq U^{(1)}$, in which case $U_1=U^{(1)}$, so the path makes a leftwards jump of magnitude $V-U_1$.
\end{enumerate}
Note that we are not concerned by the value of $U_2$, since we are interested in a left-most path. For a left-most backwards path, again at a selective event, the following events may occur:
\begin{enumerate}[(a)]
\item $V=U^{(1)}$, in which case the path jumps to the west end-point of the event, a jump which after rotation by 180 degrees becomes a rightwards jump of magnitude $V-(-r)$;
\item $V\neq U^{(1)}$, in which case the path jumps to the east end-point of the event, a jump which after rotation by 180 degrees becomes a leftwards jump of magnitude $r-V$.
\end{enumerate}
Again, because we consider a left-most path we are not concerned by the value of $U_2$.

We now compare the jumps in the (a) cases. Conditional on $V<U_1$, $V$ is uniformly distributed on $(-r,U_1)$ and thus (as in the neutral case) $V-U_1\stackrel{d}{=}V-(-r)$. Similarly, for the (b) cases, conditional on $U_1<V$, $V$ is uniformly distributed on $(U_1,r)$ and thus (also as in the neutral case) $V-U_1\stackrel{d}{=}r-V$. Thus the left-most forwards path and the rotated left-most backwards paths have the same distribution, which completes the proof.
\end{proof}

\begin{remark}
\begin{enumerate}
\item{Note that Lemma~\ref{samelaw}, with the same proof, remains true when 
the parent locations are sampled according to any symmetric distribution
on $(-r,r)$.}
\item{In previous work on the Brownian web and net, there is a strict 
self-duality in the prelimiting systems. Here, we see a new feature. 
Although separately the left and right-most paths have the 
same distributions forwards and backwards in time, their joint distribution 
differs. As can be seen in Figure~\ref{arrows}, our backwards paths
branch less frequently than forwards ones, but when they do branch, 
they make larger jumps.}
\end{enumerate}
\end{remark}

Recall the state space $\mc{K}(M)$ defined in Section \ref{intro statespace}. The space $\mc{K}(M)$ is 
an appropriate space in which to consider convergence of sets of (branching/coalescing) forwards paths, 
but it is not suitable for backwards paths. To remedy this, if 
$P$ is a set of backwards paths then we define
$-P=\{\hat{f}\-f\in P\}$, where $\hat{f}:[-\sigma_f,\infty]\to[-\infty,\infty]$
given by $\hat{f}(t)=-f(-t)$ is the rotation of $f$ by 180 degrees.
Thus, $-P\in M$ is a set of forwards paths. With a slight abuse of notation, 
if $f_n$ is a sequence of backwards paths and $f$ is a backwards path, we will say $f_n\to f$ 
in $M$ if $\hat{f}_n\to \hat{f}$ in $M$. Similarly, if $P_n$ is a sequence of sets of backwards 
paths and $P$ is a set of backwards paths we write $P_n\to P$ in $\mc{K}(M)$ to mean that 
$-P_n\to -P$ in $\mc{K}(M)$. We apply the same terminology to interpolated paths.

\subsection{Convergence of a pair of left/right paths}

We must ultimately
verify that any limit point of our combined systems of left and right-most paths
will satisfy condition $(\mathscr{B})$ of Theorem~\ref{thm:BN conv criteria}. 
As a first step, in this subsection we take the limit of a pair of paths, comprising one left-most path and 
one right-most path started at some time $s$ (which, since $\Pi^n$ is 
homogeneous in both space and time, we may, without loss of generality, take to be zero) and show that it satisfies the system~(\ref{sslrsde}). 
Our approach mirrors that in \cite{SS2008}, and as far as possible we shall adhere to their notation. 
With this in mind, let $L^n$ and $R^n$ denote respectively the left-most and right-most forwards paths associated to the points $(y^{n,l}, 0)$ and $(y^{n,r}, 0)$. We assume that the sequences of 
starting points converge to $(y^l, 0)$ and $(y^r, 0)$ respectively.
\begin{remark}
\label{coalescence_time}
A straightforward modification (in order to take
into account the selective events and the resulting drift of the left and
right-most paths) of Lemma~4.1 in 
\cite{berestycki/etheridge/veber:2013} shows that the
pair $(L^n, R^n)$ stopped when it first enters the `coalesced' state converges in
distribution to a pair of independent Brownian motions with drift $\pm\zeta$, stopped when they first meet. In particular, the first meeting times also (jointly) converge.
\end{remark}


Following \cite{SS2008}, using their Lemma~2.2, when $L_0\leq R_0$ there is a one-to-one correspondence between
weak solutions of~(\ref{sslrsde}) and solutions of the system
\begin{align}
dL_s&=\xi dB^l_{S_s}+\xi dB^c_{C_s}-\zeta ds,\label{LR1}\\
dR_s&=\xi dB^r_{S_s}+\xi dB^c_{C_s}+\zeta ds,\label{LR2}\\
s&=S_s+C_s,\label{LR3}\\
0&=\int_0^s\1\{L_s<R_s\}dC_s,\label{LR4}
\end{align}
where $B^l,B^r$ and $B^c$ are independent standard one dimensional Brownian motions. 
The infinitesimal variance $\xi^2$ of the Brownian motion and the drift $\zeta$ 
depend on $\alpha$ and $\mu$ and
are given by~\eqref{xidef} and~\eqref{zetadef} respectively. 
The solution $(L_s,R_s)$ to this system is a $C_{\R^2}[0,\infty)$ 
valued process. 

In the case $R_0<L_0$, according to \eqref{sslrsde} both $R_s$ and $L_s$ evolve as independent Brownian motions, with drift $\pm \zeta$, until they meet. Thus, in view of Remark \ref{coalescence_time}, it suffices to treat the case of $L_0\leq R_0$, where $L_0=y^l$ and $R_0=y^r$.

The essence of \eqref{LR1}-\eqref{LR4} is that, once $L_s$ and $R_s$ meet, they will accumulate non-trivial time together as a result of a sticky interaction (see Proposition 2.1 in \cite{SS2008} for details). As part of the proof of their Lemma~2.2, \cite{SS2008} show that $S_s=\int_0^s\1\{L_u<R_u\}du$ and $C_s=\int_0^s\1\{L_u=R_u\}du$. 


\begin{prop}\label{lrconv}
Let $T\in(0,\infty)$. As $n\to\infty$, $(L^n_s,R^n_s)_{s\in[0,T]}$ converges 
weakly to $(L_s,R_s)_{s\in[0,T]})$ in the sense of $D_{\R^2}[0,T]$ valued processes.
\end{prop}

The analogous result for interpolated paths, which we denote
by $(\wt{L}^n, \wt{R}^n)$, follows easily:
\begin{cor}\label{lrconv_interpolated}
Let $T\in(0,\infty)$. As $n\to\infty$, $(\wt{L}^n_s,\wt{R}^n_s)_{s\in[0,T]}$ 
converges weakly to $(L_s,R_s)_{s\in[0,T]})$ in the sense of $C_{\R^2}[0,T]$ valued processes.
\end{cor}
\begin{proof}
By Lemma~\ref{interppaths}, the weak convergence of Proposition~\ref{lrconv} 
also holds (in $D_{\R^2}[0,T]$) 
when $(L^n,R^n)$ is replaced by $(\wt{L}^n,\wt{R}^n)$. Since the space of continuous paths with the supremum topology 
is continuously embedded in the space of c\`adl\`ag paths with the Skorohod topology, it follows that the same convergence 
holds in $C_{\R^2}[0,T]$.
\end{proof}

The remainder of this subsection is devoted to the proof of Proposition~\ref{lrconv}.
We begin by breaking down the evolution of the pair $(L^n_s,R^n_s)$ into several different pieces. At time $s\geq 0$, we say $L^n_s$ and $R^n_s$ are
\begin{align*}
\textit{ coalesced}\;\text{ if }&L^n_s=R^n_s,\\
\textit{ nearby}\;\text{ if }&L^n_s\neq R^n_s\text{ and }|L^n_s-R^n_s|\leq \frac{2\mc{R}}{n^{1/2}},\\
\textit{ separated}\;\text{ if }&|L^n_s-R^n_s|>\frac{2\mc{R}}{n^{1/2}}.
\end{align*}
For $s\geq 0$ we set
\begin{align}
C^n_s&=\int_{0}^{s}\1\{L^n_u,R^n_u\text{ are coalesced}\}du,\notag\\
N^n_s&=\int_{0}^{s}\1\{L^n_u,R^n_u\text{ are nearby}\}du,\label{CNSdef}\\
S^n_s&=\int_{0}^{s}\1\{L^n_u,R^n_u\text{ are separated}\}du,\notag
\end{align}
and we note that $C^n_s+N^n_s+S^n_s=s$.

We define sequences of stopping times to track the changes of state of $(L^n,R^n)$ during $[0,T]$. Firstly,
\begin{align*}
\tau^{n,C}_1&=\inf\{s\geq 0\-(L^n_s,R_s^n)\text{ are coalesced}\};\\
\tau^{n,C}_k&=\inf\{s\geq \tau^{n,C}_{k-1}\-(L^n_s,R_s^n)\text{ are coalesced, }(L^n_{s-},R_{s-}^n)\text{ are not coalesced}\}.
\end{align*}
Similarly, we define sequences 
$\tau^{n,N}_k$ and $\tau^{n,S}_k$ for `re-entrance' times of $(L^n_r,R^n_s)$ to the states
of `nearby' and `separated' respectively.
It is easily seen that each $\tau^{n,C}_k,\tau^{n,N}_k,\tau^{n,S}_k$ is a stopping time and $(L^n,R^n)$ is strong Markov.

Each jump of $(L^n,R^n)$ is caused by one or both lineages being affected by a single event of $\Pi^n$. If $(L^n,R^n)$ is coalesced immediately before this event then the event affects both $L^n$ and $R^n$, whereas if they are separated the event affects only one of the two. 
The motion is more complicated when $(L^n,R^n)$ is in the nearby state, when events can affect one or both lineages, but we shall see that the time spent in that state is negligible as we pass to the limit. 

In order to identify the limiting objects, it is convenient to isolate the parts of the motion that 
contribute to the drift from those that contribute to the martingale terms in~(\ref{LR1}) and~(\ref{LR2}). 
The decomposition we make is not unique. Our particular choice highlights the fact that the 
martingale part of the motion of lineages is driven by neutral events, while the drift can 
be attributed to selection.

First we are going to define three random walks, from which we can build $(L^n,R^n)$ when we are
in the coalesced or separated states. To understand the origin of these, first suppose that a lineage is hit by
a neutral event. When this happens, the position, $y$, of the lineage is uniformly distributed on the 
region affected by the event
and it will jump to the position $z$ of the parent, which is also uniformly distributed on the region. 
Neutral events fall according to a Poisson Point Process with intensity
$$n^{1/2}dx\otimes n(1-\v{s}_n) dt\otimes \mu^n(dr),$$
so they hit $y$ at rate 
\begin{align}
K_n&=n(1-\v{s}_n) \int_{-\infty}^\infty\int_0^\infty \int_{-r}^r\1\{y\in [x-r,x+r]\}\,dx\,\mu^n(dr)\, n^{1/2}dx\notag\\
&=2n(1-\v{s}_n)\int_0^\infty r\mu(dr).\label{Kndef}
\end{align}
We define $V^n$ to be a symmetric random walk driven by a Poisson Point Process with 
intensity 
$$n(1-\v{s}_n)dt\otimes 2r \mu(dr).$$
At an event $(t,r)$, the walk jumps with displacement $J_1/\sqrt{n}$ where
\begin{equation}\label{Jndist}
\P\l[J_1\in A\r]=\P\l[Z_r-U_r\in A\r],
\end{equation}
and $U_r$ and $Z_r$ are independent uniform random variables on $[0,2r]$. 

Now consider the motion due to selective events. If the pair is coalesced immediately before the event,
then their position is uniformly distributed on the affected region and the left-most path will jump with
displacement $z_1-y$ and the right-most path jumps with displacement $z_2-y$ where $z_1<z_2$ are
the (uniformly distributed) positions of the two potential parents of the event. If the pair $(L^n,R^n)$ is 
separated, then only one of them will be affected by any given event. Selective events fall with
intensity 
$$n^{1/2}dx\otimes n\v{s}_ndt\otimes \mu^n(dr).$$
We define  
a random walk
$(D^{n,-}, D^{n,+})$, whose jumps are driven by a Poisson Point Process with intensity
$$n\v{s}_ndt\otimes 2r \mu(dr).$$
At an event $(t,r)$, 
$D^{n,-}$ jumps with displacement $(Z_1-Y)/\sqrt{n}$ and $D^{n,+}$ jumps with displacement 
$(Z_2-Y)/\sqrt{n}$ where 
$Z_1=\min\{U_1,U_2\}$ and $Z_2=\max\{U_1,U_2\}$ with $U_1$, $U_2$ and $Y$ independent uniformly
distributed random variables on $[0,2r]$.  

\begin{lemma}\label{Vconv}
As $n\rightarrow\infty$, $V^{n}$ converges weakly to $\xi B$ where $B$ is a standard Brownian motion and
\begin{equation*}
\xi^2 = \frac{4}{9}\int_0^{\mc{R}}r^3\mu(dr).
\end{equation*}
\end{lemma}
\begin{proof}
Evidently $J_1$ has mean zero and, conditional on $r$, its variance is $4r^2$ 
times the variance of the minimum of two independent uniform random variables 
on $[0,1]$. Thus, conditional on $r$, the variance of $J_1$ is $2r^2/9$. 
The lemma now follows from the Functional Central Limit Theorem (see, for example, \cite{ethier/kurtz:1986}, Section~7.1).
\end{proof}

Now consider $D^{n,\pm}$. 
\begin{lemma}\label{Dconv}
Let $T>0$. As $n\to\infty$, $(D^{n,-}, D^{n,+})$ converges weakly to the deterministic process $s\mapsto (-\zeta s, \zeta s)$ where 
\begin{equation*}
\zeta=\frac{2}{3}\alpha\int_0^{\mc{R}}r^2\mu(dr).
\end{equation*}
\end{lemma}
\begin{proof}Since these walks experience jumps of size ${\mathcal O}(1/\sqrt{n})$ at rate
\begin{equation}
\label{rateford}
2n\v{s}_n\int_0^{\mc{R}}r\mu(dr)=2\alpha\sqrt{n}\int_0^{\mc{R}}r\mu(dr),
\end{equation}
which is proportional to $\sqrt{n}$, we see that we have a strong law rescaling.
In the notation above, conditional on $r$,
$\E[Z_2-Y]=-\E[Z_1-Y]=\frac{r}{3}.$
By the law of large numbers as $n\rightarrow\infty$, $(D^{n,-},D^{n,+})$ converges weakly to the
deterministic process
$s\mapsto(-\zeta s,\zeta s),$
with $\zeta$ as in the statement of the lemma.
\end{proof}

When $(L^n, R^n)$ is coalesced, its jumps have the same distribution as 
$(V^n+D^{n,-}, V^n+D^{n,+})$.
When $(L^n, R^n)$ is separated,   
its jumps have the same distribution as $(V^{n,l}+D^{n,l,-},V^{n,r}+D^{n,r,+})$ where 
$V^{n,l}, V^{n,r}$ are independent copies of $V^n$ and $D^{n,l,-}$, $D^{n,r,+}$ are independent 
and with the same distribution as $D^{n,-}$, $D^{n,+}$ respectively.
When $L^n$ and $R^n$ are nearby the evolution is more complicated; in fact in this case the joint jump 
distribution depends on $|L^n-R^n|$. Happily, because of Lemma~\ref{noN} (see below), 
we will not need to describe the evolution in this case explicitly and we will denote it simply 
by $(\mc{N}^{n,l}_s,\mc{N}^{n,r}_s)$. 

Since $(L^n,R^n)$ is always in exactly one of the states `coalesced', `nearby', and `separated', and using spatial and temporal homogeneity of $\Pi^n$, it follows
from the above that we can represent the dynamics of $(L^n, R^n)$ in terms of three independent copies
of the triple $(V^n, D^{n\pm})$ which we denote $(V^{n,\alpha}, D^{n,\alpha,\pm})$ with 
$\alpha\in\{c,l,r\}$:
\begin{align}
L^n_s&=L^n_0+V^{n,l}_{S^n_s}+D^{n,l,-}_{S^n_s}+\mc{N}^{n,l}_{N^n_s}+V^{n,c}_{C^n_s}+D^{n,c,-}_{C^n_s},\label{LRpre1}\\
R^n_s&=R^n_0+V^{n,r}_{S^n_s}+D^{n,r,+}_{S^n_s}+\mc{N}^{n,r}_{N^n_s}+V^{n,c}_{C^n_s}+D^{n,c,+}_{C^n_s},\label{LRpre2}\\
s&=C^n_s+N^n_s+S^n_s,\label{LRpre3}\\
0&=\int_0^s\1\{R^n_u>L^n_u\}dC^n_u.\label{LRpre4}
\end{align}
Of course the `clocks' $(C^n, S^n, N^n)$ are coupled with the random walks $V^{n,\alpha}$ and
$D^{n,\alpha, \pm}$.

Our next task is to prove that the time spent in the `nearby' state is negligible. We require two 
preliminary estimates on the time between changes of state.  The first says that each visit to the nearby
state lasts at most $\mc{O}(1/n)$ units of time. 
The second estimate says that visits to the coalesced
state last $\mc{O}(1/\sqrt{n})$ units of time. This will limit the possible number of such visits in the time interval $[0,T]$ to be $\mathcal O(\sqrt{n})$ and since, moreover, the number of visits to the nearby state before
the pair visits the coalesced state is $\mc{O}(1)$, this in turn allows us to control the number of
visits to the nearby state. 

\begin{lemma}\label{Ntime}
Let $k\in\N$ and let $\tau'_{k'}$ be the next state change after $\tau^{n,N}_k$. Then the random variables 
$(\tau'_{k'}-\tau^{n,N}_k)_{k\in\N}$ are an independent sequence and there exists $A\in(0,\infty)$, not dependent on $k$, 
such that
$$\E\l[\tau'_{k'}-\tau^{n,N}_k\r]\leq \frac{A}{n}.$$
Further, there exists $q>0$, not dependent on $k$, such that the probability that $(L^n,R^n)$ is 
coalesced at $\tau'_{k'}$ is greater than $q$.
\end{lemma}
\begin{proof}
Independence is clear since the jumps determining the distinct $\tau'_{k'}-\tau_k^{n,N}$ are driven by disjoint collections of 
events of $\Pi^n$. 
If $(L^n, R^n)$ are nearby they must either be at a distance smaller than $\frac{3}{2}\mc{R}/{\sqrt{n}}$, or at a distance 
between $\frac{3}{2}\mc{R}/{\sqrt{n}}$ and $2\mc{R}/{\sqrt{n}}$. In the first scenario, the probability that they coalesce 
through the the next event that affects either of them is bounded away from $0$. In the second scenario, the probability 
that the next event that affects either of them brings them closer than $\frac{3}{2}\mc{R}/{\sqrt{n}}$ is also bounded 
away from $0$. 
This guarantees that the probability that $(L^n, R^n)$ is coalesced at $\tau'_{k'}$ is bounded below by some $q>0$. On 
the other hand, if the walkers are at a distance in $(\frac{3}{2}\mc{R}/\sqrt{n}, 2\mc{R}/\sqrt{n})$, the probability that
they separate at the next step is strictly positive. Thus the number of jumps until they either coalesce of separate has finite 
mean and since events affect them at rate ${\mathcal O}(n)$ the result follows.
\end{proof}

\begin{lemma}\label{Ctime}
Let $k\in\N$ and let $\tau''_{k''}$ be the next state change after $\tau^{n,C}_k$. Then the random variables $(\tau''_{k''}-\tau^{n,C}_k)_{k\in\N}$ are an i.i.d.~sequence and there exists $A'\in(0,\infty)$, not dependent on $k$, such that
$$\E\l[\tau''_{k''}-\tau^{n,C}_k\r]\geq \frac{A'}{\sqrt{n}}.$$
\end{lemma}
\begin{proof}
This is trivial, since any jump out of the coalesced state is due to a selective event and the rate
at which these occur is given by~(\ref{rateford}).
\end{proof}

\begin{lemma}\label{noN}
Fix $T>0$ and let $N_T^n$ denote the total time spent in the nearby state up to time $T$.
Then $N^n_T\to 0$ in probability as $n\rightarrow\infty$.
\end{lemma}
\begin{proof}
The idea is simple. Since $C_T\leq T$, using Lemma~\ref{Ctime}, the number of visits to the coalesced
state in $[0,T]$ has mean at most $T\sqrt{n}/A'$. But by Lemma~\ref{Ntime}, the expected number
of visits to the coalesced state is at least $q$ times the expected number of visits to the nearby state. Thus
the expected number of visits to the nearby state is at most $T\sqrt{n}/(qA')$ and since, again by Lemma~\ref{Ntime}, each has expected duration at most $A/n$, $\E[N_T]\leq TA/(qA'\sqrt{n})$ and the result
is proved. 
\end{proof}

Since $L^n$ and $R^n$ evolve as $V^n+ D^{n,-}$ and $V^n+ D^{n,+}$ respectively,
both converge individually and so their joint law is tight. Moreover, since $C^n$, $N^n$ and $S^n$ are continuous increasing processes, with
rate of increase bounded by one, their joint law is also tight.  
Evidently we now have that
$$(L_s^n,R_s^n,V_s^{n,l},V_s^{n,r},V_s^{n,c},D_s^{n,l,-},D_s^{n,r,+},D_s^{n,c,-},D_s^{n,c,+},C_s^n,N_s^n,S_s^n)_{s\geq 0}$$
is tight and by passing to a subsequence we may assume that it converges weakly to some 
limiting process
$$(L_s, R_s, \xi B^l_s, \xi B^r_s, \xi B^c_s, -\zeta s, \zeta s, -\zeta s, \zeta s, C_s, 0, S_s)_{s\geq 0},$$
where $B^l$, $B^r$ and $B^c$ are independent (by construction). Here, Lemma \ref{noN} gives that $N_s^n\to 0$.
By Skorohod's Representation Theorem, by passing to a further subsequence if necessary, we can assume
that the convergence is almost sure. We claim that the limit $(L_s, R_s, C_s, S_s)_{s\geq 0}$ 
then satsifies~(\ref{LR1}-\ref{LR4}). 
Indeed, letting $n\to\infty$ in~\eqref{LRpre1}, \eqref{LRpre2} and~\eqref{LRpre3} we obtain precisely~\eqref{LR1}, \eqref{LR2} and~\eqref{LR3}. Note that here the term $\mc{N}^{n,\dagger}_{N^n_s}$ vanishes as an easy consequence of $N_s^n\to 0$.

Obtaining~\eqref{LR4} from~\eqref{LRpre4} requires a little more work  (because the function $x\mapsto\1\{x>0\}$ is not continuous), but we need only adapt the
approach of \cite{SS2008}.
For each $\delta>0$ let $\rho_\delta$ be a continuous non-decreasing function such that $\rho_\delta(u)=0$ for $u\in[0,\delta]$ and $\rho_\delta(u)=1$ for $u\in[\delta,\infty)$. Using~\eqref{LRpre4} we have
\begin{align*}
0=\int_0^T\1\{R^n_s>L^n_s\}dC^n_s&=\int_0^T\1\l\{R^n_s-L^n_s>\frac{2\mc{R}}{n^{1/2}}\r\}dC^n_s+\int_0^T\1\l\{R^n_s-L^n_s\in\l(0,\frac{2\mc{R}}{n^{1/2}}\r]\r\}dC^n_s\\
&\geq \int_0^T\rho_\delta(R^n_s-L^n_s)dC^n_s\geq 0,
\end{align*}
provided that $\delta\geq 2\mc{R}/\sqrt{n}$. For such $n$ we thus have
$\int_0^T\rho_\delta(R^n_s-L^n_s)dC^n_s=0$
and letting $n\to\infty$ we obtain
$\int_0^T\rho_\delta(R_s-L_s)dC_s=0$
for all $\delta>0$. Letting $\delta\to 0$ we obtain
$\int_0^T\1\{R_s>L_s\}dC_s=0$
which is~\eqref{LR4}.

This completes the proof of Proposition~\ref{lrconv}.

\begin{remark}\label{lrconv_backwards}
An entirely analogous proof gives convergence of a pair of backwards right and left-most paths to left/right Brownian motions. In view of Remark \ref{coalescence_time}, this convergence occurs jointly with convergence of their first meeting time.
That the constants $\xi$ and $\zeta$ are unchanged follows from Lemma~\ref{samelaw}.
\end{remark}

\section{Spaces of {\cadlag} paths}
\label{sec:state_space}

In this section, we construct the space $\mc{K}(M)$, which is the {\cadlag} path equivalent of the state space introduced by \cite{FIN2004} for the Brownian web (and later used in \cite{SS2008} for the net). 

\subsection{Skorohod paths with different domains}
\label{Gsec}

We begin by studying the space 
$$G=\l\{g:[\sigma_g,2]\to[-1,1]\-g\text{ is {\cadlag}},\;\sigma_g\in[-1,1],\;g\text{ is constant on }[1,2]\r\}.$$
We wish to treat $G$ as a space of paths with a Skorohod-like topology, 
but since paths in $G$ can have different domains, we must extend the usual
approach. We refer to Chapter~3, Section~12 of \cite{billingsley:1995} and 
Chapter~3, Section~5 of \cite{ethier/kurtz:1986},
upon which our arguments are heavily based, 
for the standard theory of the Skorohod topology.

For $g,h\in G$, let $\Lambda'[g,h]$ denote the set of strictly increasing 
bijections from $[\sigma_g,2]\to[\sigma_h,2]$. We define $\Lambda[g,h]$ to 
be the subset of $\lambda\in\Lambda'[g,h]$ for which
\begin{equation*}
\gamma_{g,h}(\lambda)=
\sup\limits_{\sigma_g\leq t<s\leq 2}\l|\log\frac{\lambda(s)-\lambda(t)}{s-t}\r|<\infty. 
\end{equation*}
For such $g,h,\lambda$ we define
$$d(g,h,\lambda)=\sup\limits_{t\in[\sigma_g,2]}|g(t)-h(\lambda(t))|.$$
and
\begin{equation}\label{rhodef}
\rho(g,h)=\inf_{\lambda\in\Lambda[g,h]}\Big(\gamma_{g,h}(\lambda) \vee d(g,h,\lambda)\Big).
\end{equation}
Our main aim in this subsection is to show that $G$ is a complete and separable metric space under the metric
$$
d(g,h)=
\rho(g,h) \vee |\sigma_g-\sigma_h|
$$
Intuitively, this says that paths in $G$ converge if their domains converge and, as the domains become close, the paths also become close (in the Skorohod sense). 
We take $\sigma_g\in[-1,1]$ and the domain of $g\in G$ to be $[\sigma_g,2]$ 
for technical reasons: if instead we took the domain $[\sigma_g,1]$, 
$\Lambda'[g,h]$ would be empty whenever $\sigma_h<\sigma_g=1$. 
For $s\in [-1,1]$, we write $G[s]=\{g\in G:\sigma_g=s\}$.
\begin{remark}\label{justsk}
For $s\in[-1,1]$, $G[s]$ is precisely the space of {\cadlag} paths mapping $[s,2]\to[-1,1]$ that are constant on $[1,2]$. Moreover, on $G[s]$, $\rho$ coincides with the usual Skorohod metric.
\end{remark}

\begin{lemma}\label{Gdmet}
The space $(G,d)$ is a metric space.
\end{lemma}
\begin{proof}
If $d(g,h)=0$ then $\sigma_g=\sigma_h$, so by Remark~\ref{justsk} we have $g=h$. For any $\lambda\in\Lambda[g,h]$, we have $\lambda^{-1}\in\Lambda[h,g]$. Since 
$$\gamma_{g,h}(\lambda)=
\sup_{\sigma_g\leq t<s\leq 2}\l|\log\frac{\lambda(s)-\lambda(t)}{s-t}\r|
=\sup_{\sigma_h\leq t<s\leq 2}\l|\log\frac{s-t}{\lambda^{-1}(s)-\lambda^{-1}(t)}\r|
=\gamma_{h,g}(\lambda^{-1})$$ 
and, similarly, $d(g,h,\lambda)=d(h,g,\lambda^{-1})$, we have that $d$ is symmetric.

It remains to prove that $d$ satisfies the triangle inequality, for which it
suffices to show that the triangle inequality holds for $\rho$. To see this, 
take $f,g,h\in G$. For $\lambda_1\in\Lambda[f,g]$ and 
$\lambda_2\in\Lambda[g,h]$ we have $\lambda_2\circ \lambda_1\in\Lambda[f,h]$ 
and 
\begin{align}
\gamma_{f,h}(\lambda_2\circ\lambda_1)&=\sup_{\sigma_f\leq t<s\leq 2}\l|\log\frac{(\lambda_2\circ\lambda_1)(s)-(\lambda_2\circ\lambda_1)(t)}{\lambda_1(s)-\lambda_1(t)}\,\frac{\lambda_1(s)-\lambda_1(t)}{s-t}\r|\notag\\
&\leq\sup_{\sigma_g\leq t<s\leq 2}\l|\log\frac{\lambda_2(s)-\lambda_2(t)}{s-t}\r|+\sup_{\sigma_f\leq t<s\leq 2}\l|\log\frac{\lambda_1(s)-\lambda_1(t)}{s-t}\r|\notag\\
&= \gamma_{g,h}(\lambda_2)+\gamma_{f,g}(\lambda_1).\label{rhomet1}
\end{align}
Similarly, 
\begin{align}
d(f,h,\lambda_1\circ \lambda_2)&=\sup_{t\in[\sigma_f,2]}|f(t)-h(\lambda_2(\lambda_1(t)))|\notag\\
&\leq \sup_{t\in[\sigma_f,2]}|f(t)-g(\lambda_1(t))|+\sup_{t\in[\sigma_g,1]}|g(t)-h(\lambda_2(t))|\notag\\
&=d(f,g,\lambda_1)+d(g,h,\lambda_2).\label{rhomet2}
\end{align} 	
Combining~\eqref{rhomet1} and~\eqref{rhomet2} we have that $\rho(f,h)\leq \rho(f,g)+\rho(g,h)$, as required.
\end{proof}

\begin{lemma}\label{Gdsep}
The space $(G,d)$ is separable.
\end{lemma}
\begin{proof}
Let $g\in G$ and suppose $\sigma_g\in(-1,1)$. Let $(q_i)$ be an increasing 
sequence in $\Q\cap(-1,1)$ such that $q_i\uparrow\sigma_g$ and, for each $i$, 
define $\lambda_i:[q_i,2]\to[\sigma_g,2]$ by setting 
$\lambda_i(q_i)=\sigma_g$, $\lambda_i(1)=1$, $\lambda_i(2)=2$ and taking 
$\lambda_i$ to be
linear on $[q_i,1]$ and on $[1,2]$. Define $g_i\in G[q_i]$ by $g_i(t)=g(\lambda_i(t))$. Then $\lambda^{-1}\in\Lambda[g,g_i]$ and
$$\gamma_{g,g_i}(\lambda_i^{-1})=\l|\log\frac{1-q_i}{1-\sigma_g}\r|,\hspace{1pc}d(g,g_i,\lambda_i^{-1})=0.$$
Hence $d(g,g_i)\to 0$ as $i\to\infty$. By Remark~\ref{justsk}, for each $q\in\Q\cap[-1,1]$ the space $(G[q],\rho)$ is separable, hence $(G,d)$ is separable.
\end{proof}

Before we address completeness, we recall that the Skorohod topology is 
often characterized using a metric with respect to which it is not complete; 
this characterization is useful primarily because it is easier to work with. 
The extension to $G$ is as follows.

For $g,h\in G$ and $\lambda\in\Lambda'[g,h]$ define 
$\gamma'_{g,h}(\lambda)=\sup_{t\in[\sigma_g,2]}|\lambda(t)-t|$. Then, let 
$$\rho'(g,h)=\inf_{\lambda\in\Lambda'[g,h]}\l(\gamma'_{g,h}(\lambda)\vee d(g,h,\lambda)\r)$$
and define $d'(g,h)=|\sigma_g-\sigma_h|\vee\rho'(g,h).$ It can be checked, 
in similar style to the proof of Lemma~\ref{Gdmet}, that $(G,d')$ is a 
metric space.

\begin{lemma}\label{sametop}
The metrics $d$ and $d'$ generate the same topology on $G$.
\end{lemma}
\begin{proof}
First note that $\Lambda[g,h]\sw \Lambda'[g,h]$ and, since 
$|x-1|\leq e^{|\log x|}-1$ for all $x>0$, for $\lambda\in\Lambda[g,h]$ we have
\begin{multline}\label{gamma'control}
\gamma'_{g,h}(\lambda)=\sup\limits_{t\in(\sigma_g,2]}|t-\sigma_g|
\l|\frac{\lambda(t)-\sigma_g}{t-\sigma_g}-1\r|
\leq
\sup\limits_{t\in(\sigma_g,2]}|t-\sigma_g|
\left\{\l|\frac{\lambda(t)-\lambda(\sigma_g)}{t-\sigma_g}-1\r|
+\l|\sigma_h-\sigma_g\r|\right\}
\\
\leq 3\left(e^{\gamma_{g,h}(\lambda)}-1+\l|\sigma_h-\sigma_g\r|\right).
\end{multline}
(We have used that $\lambda(\sigma_g)=\sigma_h$ and 
the continuity of $\lambda$ at $\sigma_g$.)

Let $(g_n)\sw G$ and $g\in G$. If $d(g_n,g)\to 0$ then it follows readily 
from~\eqref{gamma'control} and the definitions that $d'(g_n,g)\to 0$. 
It remains to prove the converse; suppose instead that $d'(g_n,g)\to 0$.

Fix $N\in\N$. Since $d'(g_n,g)\to 0$ there exists a sequence $\lambda^N_n\in\Lambda[g_n,g]$ such that 
\begin{equation}\label{controlsametop}
\gamma'_{g_n,g}(\lambda^N_n)\vee d(g_n,g,\lambda^N_n)\vee |\sigma_{g_n}-\sigma_g|\to 0
\end{equation}
as $n\to\infty$. Define $\tau^N_0=\sigma_{g}$ and for $k=1,2,\ldots$ define
\begin{equation}\label{tauNkdef}
\tau^N_k=2\vee\inf\l\{t>\tau^N_{k-1}\-|g(t)-g(\tau^N_k)|>\frac{1}{N}\r\}
\end{equation}
up until the first $k=k_N$ for which $\tau^N_k=2$. Since $g$ is {\cadlag}, $(\tau^N_k)_{k=0}^{k_N}$ is a finite, strictly increasing sequence, and $\tau^N_{k_N}=2$.

For each $n\in\N$, define $\mu^N_n$ to be the unique piecewise linear 
function for which 
\begin{equation}\label{mupivots}
\mu^N_n(\tau^N_k)=(\lambda^N_n)^{-1}(\tau^N_k)
\end{equation}
for all $k=0,\ldots,k_N$ (and is linear in between those points). 
Then, $\mu^N_n\in \Lambda'[g,g_n]$ and, moreover,
\begin{align*}
\gamma_{g,g_n}(\mu^N_n)=\sup\limits_{k=1,\ldots,k_N}\l|\log\frac{(\lambda^N_n)^{-1}(\tau^N_{k})-(\lambda^N_n)^{-1}(\tau^N_{k-1})}{\tau^N_k-\tau^N_{k-1}}\r|<\infty
\end{align*}
so that $\mu^N_n\in\Lambda[g,g_n]$. In fact, since~\eqref{controlsametop} 
implies that $\lim_{n\to\infty}\lambda^N_n(\tau^N_k)=\tau^N_k$, we have 
$\gamma_{g,g_n}(\mu^N_n)\to 0$ as $n\to\infty$. Further, 
\begin{align*}
&\sup\limits_{t\in[\sigma_{g_n},2]}\l|g_n(t)-g\l((\mu^N_n)^{-1}(t)\r)\r|\\
&\hspace{1pc}\leq \sup\limits_{t\in[\sigma_{g_n},2]}\l|g_n(t)-g\l(\lambda^N_n(t)\r)\r| + \sup\limits_{t\in[\sigma_{g_n},2]}\l|g\l(\lambda^N_n(t)\r)-g\l((\mu^N_n)^{-1}(t)\r)\r|\\
&\hspace{1pc}\leq d(g_n,g,\lambda^N_n)+\sup\limits_{t\in[\sigma_g,2]}\l|g\l(\lambda^N_n\circ\mu^N_n(t)\r)-g(t)\r|\\
&\hspace{1pc}\leq d(g_n,g,\lambda^N_n)+\frac{2}{N}.
\end{align*}
Here, the final line follows from~\eqref{tauNkdef} and~\eqref{mupivots}. 
Hence, recalling that $d(g,g_n,\mu_n^N)=d(g,g_n,(\mu_n^N)^{-1})$, 
we have $d(g,g_n,\mu^N_n)\to 0$ as $n\to\infty$.

Combining the above with~\eqref{controlsametop}, we can choose a strictly 
increasing sequence $(n_N)_{N\in\N}$ of natural numbers such that, for 
all $n\geq n_N$,
$$\gamma_{g,g_n}(\mu^N_n)\leq \frac{1}{N},\hspace{2pc}
d(g,g_n,\mu^N_n)\leq\frac{3}{N}, \hspace{2pc}
|\sigma_{g_n}-\sigma_g|\leq \frac{1}{N}.$$
Define $\kappa_n=\mu^N_n$ for all $n\in\N$ such that $n_N\leq n < n_{N+1}$. 
Then $d(g_n,g)\leq \gamma_{g, g_n}(\kappa_n)\vee 
d(g,g_n,\kappa_n)\vee|\sigma_{g_n}-\sigma_g|$ 
so $d(g_n,g)\to 0$ as $n\to\infty$.
\end{proof}

The space $(G,d')$ is not complete (to see this, note first that by Remark~\ref{justsk} and Example 12.2 of \cite{billingsley:1995}, even the space $(G[s],d')$ is not complete). In order to prove completeness of $(G,d)$, it will be useful to note that there exists $\epsilon^\star>0$ such that for all $x\in[0,\epsilon^\star)$, we have 
\begin{equation}\label{fixEK}
e^x-1\leq 2x.
\end{equation}

\begin{lemma}\label{Gdcompl}
The space $(G,d)$ is complete.
\end{lemma}
\begin{proof}
It suffices to show that any Cauchy sequence in $(G,d)$ has a convergent subsequence. To this end, let $(g_k)$ be a Cauchy sequence in $(G,d)$. Thus $\sigma_{g_k}$ is Cauchy, which implies that $\sigma_{g_k}\to\alpha$ for some $\alpha\in[-1,1]$. 

With mild abuse of notation, we pass to a subsequence $(g_k)$ such that for all $j\geq k$ we have
$d(g_k,g_j)\leq 2^{-k}e^{-k-1}.$
Hence, there exists $\lambda_k\in\Lambda[g_k,g_{k+1}]$ such that, for all $k$,
\begin{equation}\label{control}
\gamma_{g_k,g_{k+1}}(\lambda_k) \vee d(g_k,g_{k+1},\lambda_k)\vee |\sigma_{g_k}-\alpha|\leq 2^{-k} \wedge \epsilon^\star.
\end{equation}


For each $k$, define $\wt{\lambda}_k:[-1,2]\to[-1,2]$ to be the function that is equal to $\lambda_k$ on $[\sigma_{g_k},2]$ and, if $\sigma_{g_k}>-1$, is linear in between $\wt\lambda_k(-1)=-1$ and $\wt\lambda_k(\sigma_{g_k})=\sigma_{g_{k+1}}$. Note that this means $\wt\lambda_k$ has constant gradient on $[-1,\sigma_{g_k}]$, and that $\wt\lambda_k$ is a continuous bijection of $[-1,2]$ to itself. 
Thus, we have
\begin{align}
\sup\limits_{-1\leq t\leq 2}\l|\wt\lambda_k(t)-t\r|\notag
&\leq \sup\limits_{\sigma_{g_k}\leq t\leq 2}|\lambda_k(t)-t|\notag\\
&\leq 3\left(e^{\gamma_{g_k}(\lambda_k)}-1+\l|\sigma_{g_{k+1}}-\sigma_{g_k}\r|
\right)\notag\\
&\leq 3(2^{-k+1}+2^{-k})=9\cdot 2^{-k}.\label{lambdaid}
\end{align}
Here, the second line follows from~\eqref{gamma'control}, and the final line from~\eqref{fixEK} and~\eqref{control}.

We now construct the limit of $(g_k)$. 
Define $\mu^n_k:[\sigma_{g_k},2]\to[\sigma_{g_{n+k}},2]$ and $\wt\mu^n_k:[-1,2]\to [-1,2]$ by
$
\mu^n_k=\lambda_{k+n}\circ\ldots\circ\lambda_{k+1}\circ\lambda_k,
$
and
$
\wt\mu^n_k=\wt\lambda_{k+n}\circ\ldots\circ\wt\lambda_{k+1}\circ\wt\lambda_k.
$
By~\eqref{lambdaid} we have
$$\sup\limits_{t\in[-1,2]}|\wt\mu^{n+1}_k(t)-\wt\mu^n_k(t)|\leq \sup\limits_{t\in[-1,2]}|\wt\lambda_{k+n+1}(t)-t|\leq 9\cdot 2^{-k-n}.$$
It follows that $(\wt\mu^n_k)_{n=1}^\infty$ is a Cauchy sequence in $(C_{[-1,2]}[-1,1],||\cdot||_\infty)$, and hence has a limit, which we denote by $\wt\mu_k$. Since the $\wt\mu^n_k$ are increasing, it is immediate that $\wt\mu_k(s)\geq\wt\mu_k(t)$ for $s\geq t$. 

We define $\mu_k:[\sigma_{g_k},2]\to [-1,1]$ by $\mu_k(t)=\wt\mu_k(t)$. Note that
\begin{align}
\mu_k(\sigma_{g_k})&=\lim\limits_{n\to\infty}\lambda_{k+n}\circ\ldots \circ\lambda_k(\sigma_{g_k})=\lim\limits_{n\to\infty}\sigma_{g_{k+n+1}}=\alpha,\label{limstarttime}\\
\mu_k(2)&=\lim\limits_{n\to\infty}\lambda_{k+n}\circ\ldots \circ\lambda_k(2)=2,\label{limendtime}
\end{align}
and also that, from~\eqref{lambdaid},
\begin{equation}\label{muk1}
|\mu_k(1)-1|=\lim\limits_{n\to\infty}|\mu_k^n(1)-1|\leq \lim_{n\to\infty}\sum\limits_{j=k}^{k+n}\sup\limits_{t\in[\sigma_{g_k},2]}|\lambda_j(t)-t|\leq 9\cdot2^{-k+2}.
\end{equation}
In similar style to~\eqref{rhomet1},
\begin{align}
\sup\limits_{\sigma_{g_k}\leq s<t\leq 2}\l|\log\frac{\mu^n_k(s)-\mu^n_k(t)}{s-t}\r|
&\leq\sum\limits_{j=k}^{k+n}\sup\limits_{\sigma_{g_j}\leq s<t\leq 2}\l|\log\frac{\lambda_j(s)-\lambda_j(t)}{s-t}\r|\notag\\
&=\sum\limits_{j=k}^{k+n}\gamma_{g_j,g_{j+1}}(\lambda_j)\notag\\
&\leq 2^{-k+1}.\label{A11repl}
\end{align}Here, to deduce the final line we use~\eqref{control}. Letting $n\to\infty$ we have 
\begin{align}\label{gammamuk}
\sup\limits_{\sigma_{g_k}\leq s<t\leq 2}\l|\log\frac{\mu_k(s)-\mu_k(t)}{s-t}\r|\leq 2^{-k+1}.
\end{align}
Consequently, $\mu_k$ is strictly increasing. Thus from~\eqref{limstarttime} 
and~\eqref{limendtime}, we have that $\mu_k:[\sigma_{g_k},2]\to[\alpha,2]$ is a strictly increasing bijection. In particular, it has an inverse $\mu^{-1}_k:[\alpha,2]\to[\sigma_{g_k},2]$.

The proof now follows the usual strategy.
For each $k$, we define $z_k:[-1,2]\to[-1,2]$ by
$$
z_k(t)=
\begin{cases}
g_k\circ \mu^{-1}_k(t) & t\in[\alpha,2]\\
0 & t\in[-1,\alpha).
\end{cases}
$$
We have
\begin{align*}
\sup_{t\in[-1,2]}\l|z_k(t)-z_{k+1}(t)\r|
&=\sup_{t\in[-1,2]}\l|g_k\l(\mu^{-1}_k(t)\r)-g_{k+1}\l(\lambda_k(\mu_k^{-1}(t))\r)\r|\\
&=\sup\limits_{\sigma_{g_k}\leq t\leq 2}\l|g_k(t)-g_{k+1}\l(\lambda_k(t)\r)\r|\\
&=d(g_k,g_{k+1},\lambda_k)\\
&\leq 2^{-k}
\end{align*}
Here, the second line follows by definition of $\mu_k$ and the final line 
follows by~\eqref{control}. Thus, by completeness of $\R$, 
there exists a function $z:[-1,2]\to[-1,1]$ such that 
\begin{equation}\label{zdef}
\sup_{t\in[-1,2]}|z_k(t)-z(t)|\to 0
\end{equation}
as $k\to\infty$. Since each $z_k$ is {\cadlag}, 
$z$ is {\cadlag}. By~\eqref{muk1} and the fact that each $g_k$ is constant 
on $[1,2]$, it follows from~\eqref{zdef} that $z$ is constant on $(1,2]$, 
hence by right continuity $z$ is constant on $[1,2]$.

We define $g$ by setting $\sigma_g=\alpha$ and $g(t)=z(t)$ on $[\sigma_g,2]$, and note that $g\in G$. We have shown above that $\mu_k$ is a strictly increasing bijection, hence $\mu_k\in\Lambda[g_k,g]$. Thus, from~\eqref{zdef} we have
\begin{equation*}
\sup\limits_{t\in[\sigma_{g_k},2]}|g_k(t)-g(\mu_k(t))|=\sup\limits_{t\in[\sigma_g,2]}|g_k(\mu^{-1}_k(t))-g(t)|\to 0\;\text{ as }\;k\to\infty.
\end{equation*}
Combining the above equation with~\eqref{gammamuk}, 
$\gamma_{g_k,g}(\mu_k)\vee d(g_k,g,\mu_k)\to 0$
and since $\sigma_g=\lim_{k\to\infty}\sigma_{g_k}$ we have $g_k\to g$ in $(G,d)$. 
\end{proof}

\subsection{The space $(M,d_M)$}
\label{useG}

Recall the space $M$ from~\eqref{Mdef} and the notation 
$\kappa_t=\tanh^{-1}(t)$. 
It will sometimes be useful to write $\kappa(t)=\kappa_t$.
Each $f\in M$ corresponds to some $\bar{f}\in G$, essentially through the relation~\eqref{fbar0}, that is
\begin{equation}\label{fbar0rep}
\bar{f}(t)=\frac{\tanh(f(\kappa_t))}{1+|\kappa_t|},
\end{equation}
for $t\in[\kappa^{-1}(\sigma_f),1]$. In order to treat $\bar{f}$ as an element of $G$ we specify that additionally $\bar{f}(t)=0$ for all $t\in[1,2]$. Note that $\sigma_{\bar{f}}=\kappa^{-1}(\sigma_f)$. In this section will use notation from Section \ref{Gsec} without comment. 

The map $f\mapsto \bar{f}$ naturally induces a pseudometric on $M$ through the relation
\begin{equation}\label{dMdbar}
d_M(f_1,f_2)=d(\bar{f}_1,\bar{f}_2).
\end{equation}
It follows immediately from Lemmas~\ref{Gdmet} and~\ref{Gdsep} that the set 
of equivalence classes of $M$, under $d_M$, form a separable metric space. 
Note that it is necessary to use equivalence classes, since all 
$f\in D[\infty]$ map to the same $\bar{f}\in G$. From now on we abuse 
notation slightly and write $(M,d_M)$ for the metric space of equivalence 
classes. This defines the metric $d_M$ that appeared in~\eqref{dMdef}. 


\begin{lemma}
The space $(M,d_M)$ is complete.
\end{lemma}
\begin{proof}
Let $(f_k)$ be a Cauchy sequence in $(M,d_M)$. Then $(\bar{f}_k)$ is a Cauchy sequence in $(G,d)$ and by Lemma~\ref{Gdcompl} there exists $g\in G$ such 
that $\bar{f}_k\to g$. It remains to show that there exists $f\in M$ such 
that $\bar{f}=g$, which will in turn follow immediately from~\eqref{fbar0rep} 
if we can show that
\begin{equation}\label{Mcompleq1}
|g(t)|\leq \frac{1}{1+|\kappa_t|}
\end{equation}
for all $t\in[\sigma_g,1]$.

Equation~\eqref{Mcompleq1} is readily seen; note that, by Lemma~\ref{sametop}, 
$\bar{f}_k\to g$ implies that there exists $\lambda_k\in\Lambda[g,\bar{f}_k]$ 
such that $\gamma'_{g,\bar{f}_k}(\lambda_k)\vee d(g,f_k,\lambda_k)\vee |\sigma_g-\sigma_{\bar{f}_k}|\to 0$. Therefore, 
$g(t)=\lim_{k\to\infty}\bar{f}_k(\lambda_k(t)).$
By~\eqref{fbar0rep} we have
$|\bar{f}_k(s)|\leq \frac{1}{1+|\kappa_s|}$
for all $s$, hence
\begin{equation}\label{Mcompleq2}
|g(t)|\leq \limsup_{k\to\infty}\frac{1}{1+|\kappa(\lambda_k(t))|}.
\end{equation}
We also have $\gamma'_{g,\bar{f}_k}(\lambda_k)=\sup_{t\in[\sigma_g,1]}|\lambda_k(t)-t|\to 0$. Combining this with~\eqref{Mcompleq2} proves~\eqref{Mcompleq1}.
\end{proof}

\begin{remark}\label{sametopM}
Note that $d'_M(f_1,f_2)=d'(\bar{f}_1,\bar{f}_2)$ is a pseudo-metric on $M$, with the same equivalence classes as $d_M$. Hence, by Lemma~\ref{sametop}, $d'_M$ generates the same topology on $M$ as $d_M$. 
\end{remark}

If we look at the subset $\wt{M}$ of $M$ consisting only of continuous functions, with the metric
\begin{equation}\label{FINRspace}
d_{\wt{M}}(f_1,f_2)=|\sigma_{\bar{f}_1}-\sigma_{\bar{f}_2}|\vee\sup\limits_{t\in[-1,1]}|\bar{f}_1(t\vee\sigma_{f_1})-\bar{f}_2(t\vee\sigma_{f_2})|
\end{equation}
then we recover the space of continuous paths introduced by \cite{FIN2004} (with a minor modification relating to the values of functions at $\{-\infty,\infty\}$, see the appendix of \citealt{SS2008} for details). 


We now establish the natural relationship between $M$ and $\wt{M}$, which 
mirrors the `usual' continuous embedding of spaces of continuous paths 
(with the $||\cdot||_\infty$ metric) into Skorohod spaces. Recall that
$\mc{K}(M)$ (resp.~$\mc{K}(\wt{M})$) denotes the space of all compact subsets
of $M$ (resp.$\wt{M}$). 

\begin{lemma}\label{cntsembedM}
The space $\wt{M}$ is continuously embedded in $M$.
Moreover, $\mc{K}(\wt{M})$ is continuously embedded in $\mc{K}(M)$.
\end{lemma}
\begin{proof}
Note that the first statement follows immediately from the second, so we will prove only the second statement. Recall that the topology generated by the Hausdorff metric (on $\mc{K}(M)$) depends only on the underlying topology (of $M$), and not on the underlying metric. In view of this fact and Remark~\ref{sametopM}, for the duration of this proof we take the Hausdorff metric on $\mc{K}(M)$ as that generated by $(M,d'_M)$.

Let $W_n,W$ be subsets of $\mc{K}(\wt{M})$ such that $W_n\to W$ in $\mc{K}(\wt{M})$. By Lemma~\ref{Hcompactlemma} the set 
$\mathscr{W}=W\cup\l(\bigcup_{n\in\N}W_n\r)$
is a compact subset of $\wt{M}$. A characterization of relative compactness in $\wt{M}$ is given in the proof of Lemma~4.6 of \cite{SS2008}, based on the Ascoli-Arzela Theorem (or, for a more detailed treatment, see the appendix of \cite{SSS2014}). It follows immediately from this characterization that the set $\bar{\mathscr{W}}=\{\bar{f}\-f\in \mathscr{W}\}$ is equicontinuous.

Let $\epsilon>0$. By equicontinuity, there exists $\delta>0$ such that $|s-t|\leq\delta$ implies 
\begin{equation}\label{equic}
\sup_{\bar{f}\in\bar{\mathscr{W}}}|\bar{f}(s)-\bar{f}(t)|\leq \epsilon.
\end{equation}
Without loss of generality we may choose $\delta\in(0,\epsilon)$. By definition of the Hausdorff metric, choose $N$ such that for all $n\geq N$, 
\begin{equation}\label{Hclose}
\sup\limits_{g\in W_n}\inf_{h\in W} d_{\wt{M}}(g,h)\leq\delta\;\text{ and }\;\sup\limits_{g\in W}\inf_{h\in W_n} d_{\wt{M}}(g,h)\leq \delta.
\end{equation}
By the first equality of~\eqref{Hclose}, for any $g\in W_n$ and $n\geq N$,
there exists $h\in W$ such that 
\begin{equation}\label{ghclose}
|\sigma_{\bar{g}}-\sigma_{\bar{h}}|\leq \delta\;\text{ and }\;\sup\limits_{t\in[\sigma_{\bar{g}}\vee\sigma_{\bar{h}},2]}|g(t)-h(t)|\leq \delta. 
\end{equation}
Define $\lambda_{\bar{g}}:[\sigma_{\bar{g}},2]\to[\sigma_{\bar{h}},2]$ by setting $\lambda(\sigma_{\bar{g}})=\sigma_{\bar{h}}$, $\lambda(1)=1$, $\lambda(2)=2$ and linear in between. Thus $\lambda_g\in\Lambda'[\bar{g},\bar{h}]$. For $t\in[\sigma_{\bar{g}},2]$, we have $|t-\lambda_g(t)|\leq |\sigma_{\bar{g}}-\sigma_{\bar{h}}|\leq\delta$. This implies that $\gamma'_{\bar{g}}(\lambda_{\bar{g}})\leq\epsilon$ and, using~\eqref{equic} and~\eqref{ghclose}, that
\begin{align*}
|\bar{g}(t)-\bar{h}(\lambda_{\bar{g}}(t))|&\leq |\bar{g}(t)-\bar{g}(\lambda_{\bar{g}}(t))|+|\bar{g}(\lambda_{\bar{g}}(t))-\bar{h}(\lambda_{\bar{g}}(t))|\leq 2\epsilon.
\end{align*}
Thus, $d'(\bar{g},\bar{h})\leq 2\epsilon$. Similarly, using the
second equality of~\eqref{Hclose}, for any $g\in W$ and $n\geq N$,
there exists $h\in W_n$ such that $d'(\bar{g},\bar{h})\leq 2\epsilon$. We thus have, for all $n\geq N$,
\begin{equation}\label{Hclose2}
\max\l(\sup\limits_{g\in W_n}\inf_{h\in W} d'(\bar{g},\bar{h}),\sup\limits_{g\in W}\inf_{h\in W_n} d'(\bar{g},\bar{h})\r)\leq 2\epsilon.
\end{equation}
Hence, $W_n\to W$ in $\mc{K}(M)$ as $n\rightarrow\infty$.
\end{proof}

In the interests of brevity, we limit our further development of the space $(M,d_M)$ to the following two results.

\begin{lemma}\label{interppaths2}
Let $f,g\in M$ with $\sigma_f=\sigma_g$ and $\sup_{t\in[\sigma_f,\infty]}|f(t)-g(t)|\leq r$. Then $d'_M(f,g)\leq r$.
\end{lemma}
\begin{proof}
Since $\sigma_f=\sigma_g$, the identity function $\iota$ is an element of $\Lambda[\bar{f},\bar{g}]$. Note that $\gamma_{\bar{f}}(\iota)=0$. Hence 
$d'_M(f,g)\leq \sup_{t\in[\sigma_{\bar{f}},\infty]}|\bar{f}(t)-\bar{g}(t)|\leq  \sup_{t\in[\sigma_f,2]}|f(t)-g(t)|\leq r,$
as required.
\end{proof}

\begin{lemma}\label{DMpartial}
Let $(f_m)\sw M$ and $f\in M$ with $\sigma_{f_m}=\sigma_f\in(-\infty,\infty)$. Then $d_M(f_m,f)\to 0$ if (the restrictions of) $f_{m}\to f$ in $D_{[\sigma_f,T]}(\R)$ for all $T\in(\sigma_f,\infty)$.
\end{lemma}
\begin{proof}
Note that $\tanh:[-\infty,\infty]\to[-1,1]$ is a contraction. For $T\in(0,\infty)$, set $\bar{T}=\tanh(T)$, and note that $\tanh$ restricted to $[-T,T]\to[-\tanh(T),\tanh(T)]$ is bi-Lipschitz. Hence there are constants $C_T\in(0,\infty)$ such that 
\begin{equation}\label{tanhcontrol}
\sup\limits_{-\bar{T}\leq t<s\leq \bar{T}}\l|\frac{\kappa_s-\kappa_t}{s-t}\r|\leq C_T,\hspace{2pc}\sup\limits_{-\infty\leq s<t\leq \infty}\l|\frac{s-t}{\kappa_s-\kappa_t}\r|\leq 1.
\end{equation}

Let $\epsilon>0$. Let $T\in(\sigma_f\vee 0,\infty)$ be such that $\frac{1}{1+T}\leq \epsilon$. Thus,
\begin{equation}\label{Tbound}
\sup\limits_{g\in M}\sup\limits_{t\in[\bar{T}\vee\sigma_{g},2]}|\bar{g}(t)|\leq\epsilon.
\end{equation}
By Theorem~12.1 of \cite{billingsley:1995} there exists $M\in\N$ such that for all $m\geq M$ there exists a continuous strictly increasing $\lambda_m:[\sigma_f,T]\to[\sigma_f,T]$ with
\begin{equation}\label{kappacontrol}
\sup\limits_{t\in[\sigma_f,T]}|t-\lambda_m(t)|\leq \frac{\epsilon}{C_T},\hspace{3pc}\sup\limits_{t\in[\sigma_f,T]}|f(t)-f_m(\lambda_m(t))|\leq\epsilon.
\end{equation}
Define $\bar{\lambda}_m:[\sigma_{\bar{f}},\bar{T}]\to [\sigma_{\bar{f}},\bar{T}]$ by $\bar{\lambda}_m(t)=\kappa^{-1}\circ\lambda_m\circ\kappa(t)$, and note that by~\eqref{kappacontrol} for all $t\in[\sigma_{\bar{f}},\bar{T}]$,
\begin{equation}\label{kappalip1}
|t-\bar{\lambda}_m(t)|\leq C_T|t-\lambda_m(t)|\leq \epsilon,
\end{equation}
and, by the right hand side of~\eqref{tanhcontrol},
\begin{align}\label{kappalip2}
|\bar{f}(t)-\bar{f}_m(\bar{\lambda}_m(t))|&=\frac{1}{1+|\kappa_t|}|\tanh(f(\kappa_t))-\tanh(f_m(\kappa(\bar{\lambda}_m(t)))|\notag\\
&\leq |f(\kappa_t)-f_m(\lambda_m(\kappa_t))|\leq \epsilon
\end{align}
Extend $\bar{\lambda}_m$ to $\eta_m:[\sigma_{\bar{f}},2]\to [\sigma_{\bar{f}},2]$ by setting $\bar{\lambda}_m(t)=t$ for $t\geq \bar{T}$. Then, combining~\eqref{Tbound}, \eqref{kappalip1} and~\eqref{kappalip2} we obtain
$\gamma'_{\bar{f}}(\bar{\lambda}_m)\leq\epsilon$
and 
$d'(\bar{f},\bar{f}_m,\bar{\lambda}_m)\leq 2\epsilon.$
It follows that $d'(\bar{f},\bar{f}_m)\to 0$, so the stated result now follows by Lemma~\ref{sametop}.
\end{proof}

\section{Convergence to the Brownian Net}
\label{sec:BN conv}

\subsection{Compactness}
\label{compactsec}

In order to use Theorem~\ref{thm:BN conv criteria}, we must verify that our 
various set of paths really are subsets of $\mc{K}(M)$. That is, we need to 
show that they are compact subsets of $\wt{M}$, which is the content of this 
subsection. We concentrate on forwards paths; analogous arguments apply to 
backwards paths.

We require three preparatory lemmas. The first two of these embody the key features of the argument; at any given time, within bounded intervals of space, the number of ancestral lines at distinct spatial locations is finite, and there do not exist ancestral lines that move arbitrarily fast across space. 

\begin{lemma}\label{finitein}
Let $n\in\N$. Then, almost surely, for all (random) $a,b,t\in\R$ the set $E_{a,b,t}=[a,b]\cap\{f(t)\-f\in \wt{\mc{P}}_n^{\uparrow}(\mc{D}_n)\}$ is finite.
\end{lemma}
\begin{proof}
For fixed deterministic $a,b,t$, it is easily seen that the set of $(y,s)$ for which $s\leq t$, $y\in[a,b]$ and the line $\{y\}\times[s,t]$ is not affected by any reproduction events, is almost surely bounded (in $\R^2$). Thus, since $\mc{D}_n$ is locally finite, the set $E_{a,b,t}$ is almost surely finite. Take countably dense (deterministic) sequences $(a_m),(b_m)$ and $(t_m)$ in $\R$; thus almost surely, for all $m_1,m_2,m_3\in\N$ for which $a_{m_1}<b_{m_2}$, the set $E_{a_{m_1},b_{m_2},t_{m_3}}$ is finite. The stated result now follows from Lemma~\ref{upsilonlemma}.
\end{proof} 

\begin{lemma}\label{pppisnice}
Let $n\in\N$. Almost surely, there does not exist a (random) sequence $(x_m,t_m,r_m)_{m=1}^\infty\sw \Pi_n$ such that $\sup_m|t_m|<\infty$, $\lim_{m\to\infty}x_m=\infty$ and $\sup_m|x_{m+1}-x_{m}|\leq 4\mc{R}_n$.
\end{lemma}
\begin{proof}
Let $K\in(0,\infty)$. Then, the probability that 
$$\Pi_n\cap \Big([4k\mc{R}_n,4(k+1)\mc{R}_n]\times [-K,K] \times [0,\mc{R}_n]\Big)=\emptyset$$ 
is positive and does not depend on $k$. Consequently, the probability that there exists a sequence $(x_m,t_m,r_m)_{m=1}^\infty\sw \Pi_n$ such that $\lim_{m\to\infty}x_m=\infty$ and $\sup_m|x_{m+1}-x_{m}|\leq 4\mc{R}_n$, with $\sup_m|t_m|\leq K$, is zero. Since $K$ was arbitrary, the result follows.
\end{proof}

Recall that our ultimate goal is to prove Theorem~\ref{result dimension one}, which claims convergence in distribution. In view of this, from now on we will (abuse notation slightly and) assume that the conclusions of Lemma~\ref{pppisnice} and \ref{finitein} hold \textit{surely}. 

The next lemma asserts that any convergent sequence of paths that becomes close, in space, to touching $\infty$ within some bounded interval of time, must converge to a constant path at $\infty$. 

\begin{lemma}\label{hittheside}
Let $n\in\N$. Let $(f_m)_{m=1}^\infty\sw \mc{P}_n^{\uparrow}(\mc{D}_n)$ be a sequence of paths and suppose $\sigma(f_m)$ converges to $v\in[-\infty,\infty]$. 
Suppose also that there exists a bounded sequence $(t_m)$ with $t_m\geq \sigma(f_m)$ for which $f_m(t_m)\to\infty$ as $m\to\infty$.

Let $f_{\infty}\in M$ be the path defined by $\sigma(f_\infty)=v$ and $f_\infty(s)=\infty$ for all $s\in[v,\infty]$.
Then $f_m\to f_\infty$ in $M$. 

Moreover, if instead $(f_m)\sw \wt{\mc{P}}_n^{\uparrow}(\mc{D}_n)$, then under the same hypothesis $f_m\to f_\infty$ in $\wt{M}$.
\end{lemma}
\begin{proof}
Fix $K\in(0,\infty)$, large enough that $\sup_m|t_m|\leq K$. Define $x(m^*,K)=\inf\{f_m(s)\-m\geq m^*, \sigma(f_m)\leq s, |s|\leq K\}$. 

Suppose, for a contradiction, that $x(m^*,K)$ does not tend to $\infty$ as $m^*\to\infty$. Then, there exists $X\in(-\infty,\infty)$ and infinitely many $m^*$ for which $x(m^*,K)\leq X$. For all such $m^*$ we have some $m\geq m^*$ and $|s|\leq K$ such that $f_m(s)\leq X$, and (by our hypothesis) as $m^*\to\infty$ we have also $f_m(t_m)\to\infty$; since $f_m\in\mc{P}^{\uparrow}(\mc{D}_n)$ this is a contradiction to Lemma~\ref{pppisnice}.

So, $x(m^*,K)\to\infty$ as $m^*\to\infty$. Thus, for any $K,X\in(-\infty,\infty)$ we can find $m^*\in\N$ such that, for all $m\geq m^*$ and $s\geq \sigma(f_m)$ such that $|s|\leq K$, we have $f_m(s)\geq X$. With this in hand, the stated results follow easily from~\eqref{dMdbar} and~\eqref{FINRspace}.
\end{proof}

Recall that, by~\eqref{Paugment}, the path $f_\infty$ in the statement of 
Lemma~\ref{hittheside} is an element of both $\mc{P}^{\uparrow}(\mc{D}_n)$ and $\wt{\mc{P}}^{\uparrow}(\mc{D}_n)$. Recall also that, in the notation of~\eqref{Paugment}, both these sets also contain paths that are infinite extenders.

\begin{lemma}\label{compact}
Let $\star\in\{\uparrow,\downarrow\}$, $\dagger\in\{l,r\}$ and $n\in\N$. Then, $\mc{P}_n^{\star}(\mc{D}_n)$ and $\mc{P}_n^{\star,\dagger}(\mc{D}_n)$ are compact subsets of $M$, also $\wt{\mc{P}}_n^{\star}(\mc{D}_n)$ and $\wt{\mc{P}}_n^{\star,\dagger}(\mc{D}_n)$ are compact subsets of $\wt{M}$.
\end{lemma}
\begin{proof}
As usual, it suffices to consider the case of forward paths. Since Lemma~\ref{hittheside}, on which the following argument relies, holds in both $M$ and $\wt{M}$, it will suffice to consider only cases in $M$. Moreover, the arguments required the case of $\mc{P}_n^{\uparrow,\dagger}(\mc{D}_n)$ are essentially identical to those required for the case $\mc{P}_n^\uparrow(\mc{D}_n)$; thus, we aim to show that $\mc{P}_n^\uparrow(\mc{D}_n)$ is a sequentially compact subset of (the metric space) $M$.

Let $(f_m)_{m=1}^\infty\sw \mc{P}_n^{\uparrow}(\mc{D}_n)$ be a sequence of paths. We must show that $(f_m)$ has a convergent subsequence, with limit in $\mc{P}_n^{\uparrow}(\mc{D}_n)$. 

We now split into several cases. 

\begin{enumerate}
\item If $(\sigma(f_m))_{m\geq 1}$ has a subsequence that converges to $\infty$ then, along this subsequence, $f_m$ converges to the degenerate path $f$ with $\sigma_f=\infty$ and $f(\infty)=0$.

\item If $(\sigma(f_m))_{m\geq 1}$ has a bounded subsequence, then consider the sequence $x_m=f_m(\sigma_{f_m})$. 
\begin{enumerate}
\item
If $(x_m)_{m\geq 1}$ is bounded, then since $\mc{D}_n$ is locally finite there 
must be a subsequence along which $(\sigma(f_m),x_m)$ is eventually constant. 
Any given ancestral line moves to one of at most two 
locations in a reproduction event, thus $(f_m)_{m\geq 1}$ 
has a convergent subsequence; to construct the limit path we successively 
follow parent points that were followed by infinitely many of our $f_m$.

\item If $(x_m)_{m\geq 1}$ is not bounded, then without loss of generality we pass to a subsequence and assume that both $x_m\to\infty$ and $\sigma(f_m)$ converges. It then follows immediately from Lemma~\ref{hittheside} that $f_m$ converges along this subsequence.
\end{enumerate}

\item If $(\sigma(f_m))_{m\geq 1}$ has a subsequence that converges to $-\infty$, then pass to that subsequence and set $t=\sup_m \sigma(f_m)<\infty$. 
\begin{enumerate}
\item If $\{m\-|f_m(t)|\leq K\}$ is finite for all $K<\infty$, then essentially the same argument as in 2(b), reliant on Lemma~\ref{hittheside}, shows that $f_m$ has a convergent subsequence.

\item If, for some $K<\infty$ the set $\{m\-|f_m(t)|\leq K\}$ is infinite, then pass to the subsequence of $f_m$ such that $|f_m(t)|\leq K$. By Lemma~\ref{finitein}, the set $[-K,K]\cap\{f_m(t)\-m\in\N\}$ is finite. Hence, there is some $|z|\leq K$ through which infinitely many $f_m$ pass. Using the same method as in \textit{2(a)}, we can construct a path $f:[t,\infty]$ with $f(t)=z$ that is followed by infinitely many $f_m$. So, pass to a further subsequence and assume $f_m(s)=f(s)$ for all $s\geq t$.

We extend $f$ backwards in time as follows. From location $(z,t)$, look backwards in time until the most recent reproduction event (strictly) before $t$ that affected $z$, say $p=(x,t',r)$. By Lemma~\ref{finitein}, the set $\{f_m(t'-)\-p\text{ affects }f_m, m\in\N\}$ is finite. Pick some element $z'$ of this set, and restrict to $f_m$ for which $f_m(t'-)=z'$. Set $f(s)=z$ for $s\in[t',t)$ set $f(s)=z$. Then, look back from $(z',t')$, and repeat (in the language of~\eqref{Paugment}, $f$ is an `infinite extender'). Thus, a subsequence of $(f_m)_{m\geq 1}$ converges to $f$. 
\end{enumerate}
\end{enumerate}

Since $(\sigma(f_m))_{m\geq 1}$ must have a subsequence that converges in 
$[-\infty,\infty]$, at least one of the above cases occurs. This completes the proof.
\end{proof}

\subsection{Convergence of multiple left/right paths}
\label{multpathssec}

We now extend Proposition~\ref{lrconv} to larger collections of left and right-most paths. Let $N\in\N$. 
Given a finite set 
$D=\{(y_i,s_i)\in\R^2\-i=1,\ldots,N\}$
of distinct points in $\R^2$ and a function 
$O:\{1,\ldots, N\}\to\{l,r\},$
\cite{SS2008}, Section~2.2, construct a system of left-right coalescing Brownian motions started from
the points of $D$. (Recall that two left-most paths coalesce on meeting, as do two right-most paths, and that~\eqref{sslrsde} describes the interaction between left-most and right-most paths.)
We write
\begin{align}\label{forwardsBMs}
\mc{P}^\uparrow(D,O)=\{B^{\uparrow,(y_i,s_i)}\-i=1,\ldots,N\}
\end{align}
for this system, where $B^{\uparrow,(y_i,s_i)}$ denotes the path of a Brownian motion started from $(y_i,s_i)$
with diffusion constant $\xi^2$ and drift $\zeta$ to the left if $O(i)=l$ and to the right if $O(i)=r$.

For each $i=1,\ldots,N$ let $d_i^n=(y_i^n,s_i^n)\in\R^2$ be such that $d_i^n\to d_i=(y_i,s_i)\in\R^2$. 
Set $D^{(n)}=\{d^n_i\-i=1,\ldots,N\}$.
We define the set $\mc{P}_n^\uparrow(D^{(n)},O)=\{f^\uparrow_1,\ldots,f^\uparrow_N\}$ 
where $f^\uparrow_i$ is the $O(i)$-most forward path from $d_i^n$ driven by events in $\Pi^n$s 
and
$\wt{\mc{P}}_n^\uparrow(D^{(n)},O)$ for the corresponding space of interpolated paths.
Note that both these sets are random elements of the product space $M^N$.

\begin{lemma}\label{multforwconv}
Let $N\in\N$ and let $D^{(n)},D,O$ be as above. Then, as $n\to\infty$, $\mc{P}^{\uparrow}_n(D^{(n)},O)$ converges weakly in $M^N$ to 
$\mc{P}^\uparrow(D,O)$ and $\wt{\mc{P}}^{\uparrow}_n(D^{(n)},O)$ converges weakly in 
$\wt{M}^N$ to $\mc{P}^\uparrow(D,O)$.
\end{lemma}
\begin{proof}
The argument is essentially identical to that of the proof of Proposition~5.2 of \cite{SS2008}. 
The construction of $\mc{P}^\uparrow(D,O)$ in Section 2.2 of \cite{SS2008} is an inductive construction that 
views $\mc{P}^\uparrow(D,O)$ as made up of several independent pieces consisting of segments of either single 
left-most paths, single right-most paths or a pair of left/right paths. The same inductive construction breaks 
down $\mc{P}^\uparrow_n(D^{(n)},O)$ into corresponding pieces. 
The stopping times used in this construction are continuous functionals on $M^N$ with respect to the law of 
independent evolutions of paths within each such piece, so the first part of the lemma follows from 
Proposition~\ref{lrconv} and Lemma~\ref{DMpartial}. Similarly, $\wt{\mc{P}}^\uparrow_n(D^{(n)},O)$ converges weakly to $\wt{\mc{P}}^\uparrow(D,O)$.
\end{proof}

\subsection{Proof of Theorem~\ref{result dimension one}}

We complete the proof of Theorem~\ref{result dimension one} in three steps. Recall that the Brownian net is denoted by $\mc{N}$, and recall the function $\mc{H}_{cross}$ defined in Section \ref{bw bn sec}.

\begin{lemma}\label{interp net conv}
As $n\to\infty$, we have that 
$$\mc{H}_{cross}\l(\wt{\mc{P}}_n^{\uparrow,l}(\mc{D}_n)\cup\wt{\mc{P}}_n^{\uparrow,r}(\mc{D}_n)\r)\to\mc{N},$$ 
in distribution in $\mc{K}(\wt{M})$. 
\end{lemma}
\begin{proof}
We verify the conditions $(\mathscr{A})$-$(\mathscr{D})$ of Theorem~\ref{thm:BN conv criteria}. This theorem is applied with $X_n^\dagger=\wt{\mc{P}}_n^{\uparrow,\dagger}(\mc{D}_n)$ and $\hat{X}_n^\dagger=\wt{\mc{P}}_n^{\downarrow,\dagger}(\mc{D}_n)$, where $\dagger\in\{l,r\}$. By Lemma~\ref{compact}, all these sets of paths are (almost surely) elements of $\mc{K}(\wt{M})$ (after rotation by
$180$ degrees about $(0,0)$ for the backwards paths). We define
$X_n=\mc{H}_{cross}(\wt{\mc{P}}_n^{\uparrow,l}(\mc{D}_n)\cup\wt{\mc{P}}_n^{\uparrow,r}(\mc{D}_n))$
and similarly for $\hat{X}_n$.

We now check the conditions in turn.
For $(\mathscr{A})$, the required statements about non-crossing paths are precisely the content of Lemma~\ref{noncrossing}.
For $(\mathscr{B})$,  the required convergence of multiple left/right paths to left/right Brownian motions is precisely the content of Lemma~\ref{multforwconv}.

We now move on to $(\mathscr{C})$. If $(x_{n,1},x_{n,2})\to (x_1,x_2)$ and $\hat{l}_n$,$\hat{r}_n$ are respectively elements of $\wt{\mc{P}}_n^{\downarrow,l}(\mc{D}_n),\wt{\mc{P}}_n^{\downarrow,r}(\mc{D}_n)$ started at $(x_{n,1},x_{n,2})$, then it follows by combining Lemmas~\ref{lrconv_interpolated} and Remark~\ref{lrconv_backwards} that $(\hat{l}_n,\hat{r}_n)\to (\hat{l},\hat{r})$ in distribution, where $(\hat{l},\hat{r})$ are a pair of left/right Brownian motions. That the first meeting time of $\hat{l}_n$ with $\hat{r}_n$ also converges (jointly) in distribution to the first meeting time of $\hat{l}$ with $\hat{r}$ follows from Remark~\ref{coalescence_time}. 

It remains to verify $(\mathscr{D})$. By Lemma \ref{noncrossing}, left-most fowards and left-most backwards paths cannot cross, and similarly for right-most paths. Therefore, a path of $X_n$ that enters a wedge $W$ of $\hat{X}_n$ from the outside must enter through the southern-most point of the wedge (i.e.~precisely where the two paths $\hat{r}$ and $\hat{l}$ defining $W$ meet, in Figure \ref{wedge}). Fix a wedge $W$ of $\hat{X}_n$ and denote this event by $E_W$. 

For the event $E_W$ to occur, some reproduction event must have a potential parent situated at the spatial location of the southern-most point of $W$. The distribution of the spatial location of a pre-parent (in a reproduction event occurring at given time) has no atoms, and thus the distribution of the spatial location of the meeting point of $\hat{r}$ and $\hat{l}$ also has no atoms; hence almost surely $E_W$ does not occur. 

The set $\hat{X}_n$ contains countably many paths and thus has at most countably many wedges. Hence, almost surely the event $E_W$ does not occur for any wedge $W$ of $\hat{X}_n$. Without loss of generality, we may assume this does not occur \textit{surely}, so as $(\mathscr{D})$ holds.
\end{proof}

\begin{lemma}\label{interp net conv 2}
As $n\to\infty$, we have that $\wt{\mc{P}}_n^\uparrow(\mc{D}_n)$ tends in distribution to $\mc{N}$
\end{lemma}
\begin{proof}
If a left-most and a right-most path of $\wt{\mc{P}}_n^\uparrow(\mc{D}_n)$ cross, then the point at which they cross much be within one of the sets $\mathscr{B}_{\Upsilon(p)}(p)$ (defined in Lemma \ref{upsilonlemma}) associated to a selective event $p\in\Pi_n$. If an interpolated arrow finishes at $p$, then by definition there are interpolated arrows ending at both potential parents of $p$. Consequently, there is a one to one correspondence, $f\mapsto f'$ between paths $f\in\mc{H}_{cross}\l(\wt{\mc{P}}_n^{\uparrow,l}(\mc{D}_n)\cup\wt{\mc{P}}_n^{\uparrow,r}(\mc{D}_n)\r)$ and paths $f'\in\wt{\mc{P}}^\uparrow(\mc{D}_n)$ such that 
\begin{equation}\label{eq:Xn_Pn}
|f(t)-f'(t)|\leq 2\mc{R}_n
\end{equation}
for all $t\geq \sigma_f=\sigma_{f'}$, and $f(\sigma_f)=f'(\sigma_{f'})$.

By Lemma \ref{compact}, we have that $\wt{\mc{P}}_n^\uparrow(\mc{D}_n)$ is an element of $\mc{K}(M)$. Combined with \eqref{eq:Xn_Pn} this means that $\wt{\mc{P}}_n^\uparrow(\mc{D}_n)$ has the same limit (in distribution) as $\mc{H}_{cross}(\wt{\mc{P}}_n^{\uparrow,l}(\mc{D}_n)\cup\wt{\mc{P}}_n^{\uparrow,r}(\mc{D}_n))$ as $n\to\infty$. Thus, from Lemma \ref{interp net conv}, $\wt{\mc{P}}_n^\uparrow(\mc{D}_n)$ tends in distribution to $\mc{N}$.
\end{proof}

To finish, we must upgrade the result of Lemma~\ref{interp net conv 2} from 
$\mc{K}(\wt{M})$ to $\mc{K}(M)$, and use {\cadlag} paths in place of interpolated paths. 

By Lemma~\ref{compact} we have that $\wt{\mc{P}}_n^{\uparrow}(\mc{D}_n)\in\mc{K}(M)$ for all $n$. 
Combining this fact with Lemma~\ref{interppaths}, and noting that there 
is a one to one correspondence between paths and interpolated paths, we obtain from Lemma \ref{interp net conv 2} that also $\mc{P}_n^{\uparrow}(\mc{D}_n)\to\mc{N}$ in distribution in 
$\mc{K}(M)$. This completes the proof of Theorem~\ref{result dimension one}.

\section{Simulations}
\label{numerics}

In all the work described in Section~\ref{intro d=1}, lineages coalesce 
instantly on meeting and, in particular, they cannot `jump over' one 
another. Our result also requires this property, which is 
achieved by setting $\upsilon=1$. 
In the absence of selection, it is shown in \cite{berestycki/etheridge/veber:2013} that
if instead we fix $\upsilon\in (0,1)$, the scaling limit of the paths 
relating a 
finite sample from the population is a system of coalescing Brownian motions, 
but with `clock' rate $\upsilon$ (so that the whole process is slowed down).
In particular, we obtain a simple time-change of the limit for $\upsilon=1$. 

It is natural to ask what happens when $\upsilon<1$ in the presence of selection. Our method of proof certainly breaks down. In particular, 
we can no longer trace the left/right most paths starting from a single point in isolation: because paths can now cross, the current left-most path may not be affected by an event,
whereas another path in the same region is, and if the parent (or one of the potential parents) of the event is to the left of the current left-most path, a new 
line of descent takes over as left-most. 

We have been unable to find a rigorous result in this context and so in this brief section, instead, we present the results of a numerical experiment.

\begin{figure}[h]
   \centering
    \includegraphics[width=16cm]{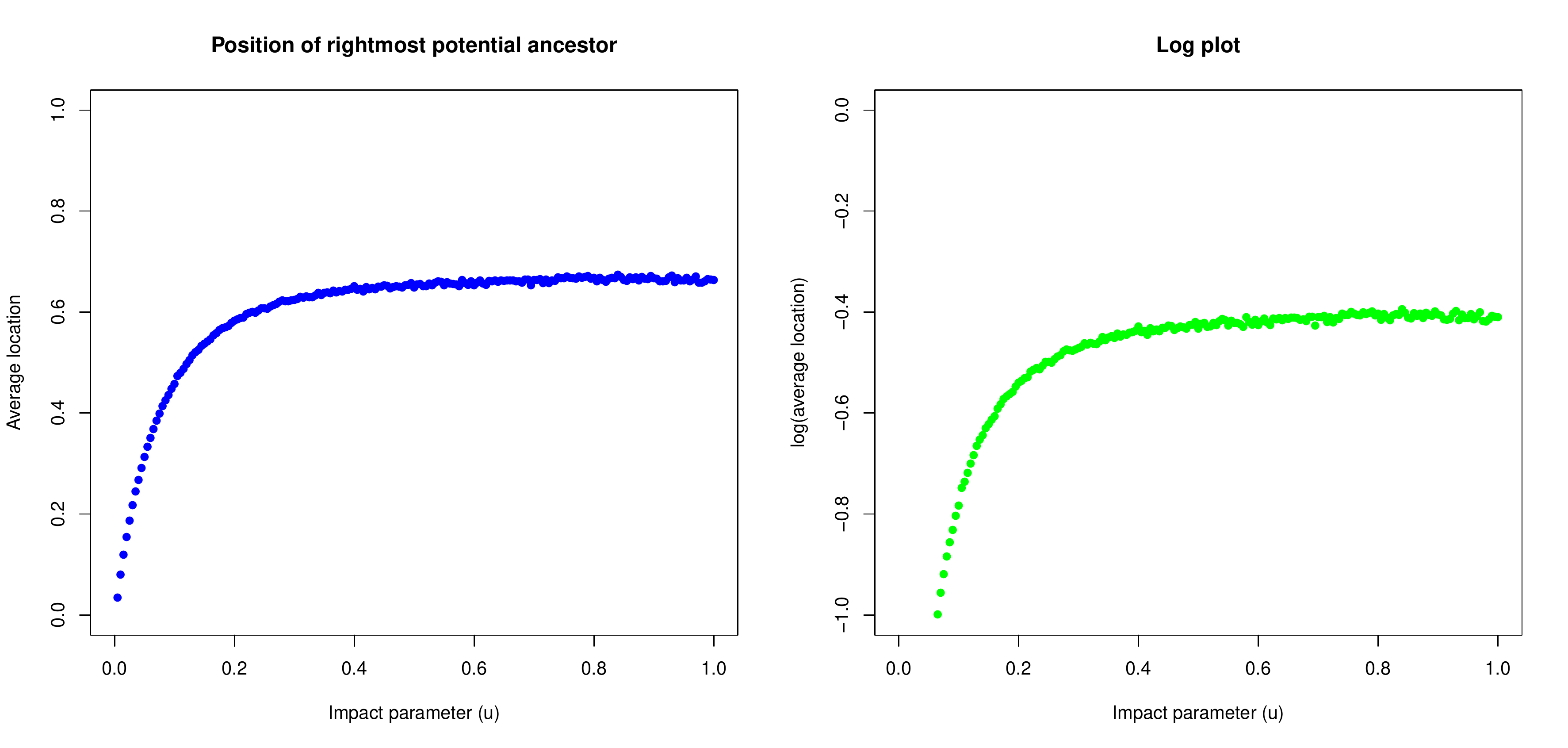}
    \caption{\textbf{Right-most potential ancestor.} The plot on the left shows $\mathscr{P}(u)$, the average location of the right most ancestor at time $1$ 
of an individual located at the origin at time $0$. Here we take $n=1000$ and 
$\mu(dr)=\delta_1$. The plot contains $200$ data points, each corresponding to 
a value of $\upsilon\in(0,1]$, spaced evenly along the horizontal axis. 
Each data point is the mean of $2000$ independent simulations with the 
corresponding value of $\upsilon$. The plot on the right is a logarithmic plot of the same data. } 
\label{Pu} 
\end{figure} 

In order to approximate the limiting process, we simulated the system of branching-coalescing lineages (which was introduced in Section~\ref{scaling}). 
In particular, we were interested in $\mathscr{P}(\upsilon)$, 
the expected position of the right-most ancestor at time $1$ of a single 
particle, which starts at the origin at time $0$. By symmetry the same 
analysis applies to the left-most particle.

As a result of the discussion above, and the result of 
\cite{berestycki/etheridge/veber:2013}, it is natural to ask if 
$\mathscr{P}(\upsilon)$ varies linearly with $\upsilon$. It seems that this 
is not the case, as is shown in the left hand plot of Figure~\ref{Pu}. 
Note that this does not remove the possibility that, as $n\to 0$, the {\slfvs} rescales to 
a Brownian net, but it does imply that the speed of such a limiting net would 
not match the speed suggested by the simple time-change in the Brownian web limit of \cite{berestycki/etheridge/veber:2013}.
As can be seen from the right hand plot of Figure \ref{Pu}, $\mathscr{P}(\upsilon)$ is also 
not of the form $\upsilon^{-\alpha}$. 

Simulating the S$\Lambda$FVS when $\upsilon=1$ is a quite different task 
from simulating it when $\upsilon$ is close to $0$. In the former case, 
it is more efficient to generate reproduction events by simulating the 
underlying Poisson Point Process, whereas in the latter case it is more 
efficient to track clusters of particles that are close enough to be affected 
by the same event and simulate their (correlated) motion directly. 
Our simulation employs both methods of event sampling and alternates 
between them based on which method is asymptotically more efficient 
for the given value of $\upsilon$. 

In practice, for smaller values of $\upsilon$ (i.e.~closer to $0$) 
larger values of $n$ are needed to accurately simulate the {\slfvs}. It seems 
possible that, for $\upsilon$ closer to $0$, 
the numerics suggested by Figure~\ref{Pu} are not a reflection of the 
true behaviour as $n\to\infty$.

The 
{\CC} code which generated the data displayed in Figure~\ref{Pu} can 
be obtained from \url{http://www.github.com/nicfreeman1209}.

\appendix

\section{On compactness}

In the proof of Lemma~\ref{cntsembedM} we used following result, which is almost certainly known but for which we were unable to find a reference.

\begin{lemma}\label{Hcompactlemma}
Let $(\mathscr{M},d_{\mathscr{M}})$ be a metric space and let $\mc{K}(\mathscr{M})$ be the 
Hausdorff space of compact subsets of $\mathscr{M}$. Let $(W_n)_{n\in\N}$ 
be a sequence in $\mc{K}(\mathscr{M})$ such that 
$W_n\to W_\infty\in\mc{K}(\mathscr{M})$ as $n\to\infty$. Then 
$\mathscr{W}=W_\infty\cup\l(\bigcup_{n\in\N}W_n\r)$ is a compact subset of $M$.
\end{lemma}
\begin{proof}
Since $\mathscr{M}$ is a metric space, a subset $W$ of $\mathscr{M}$ is 
compact if and only if $W$ is sequentially compact. Let $(w_n)_{n\in\N}$ 
be any sequence in $\mathscr{W}$ and define 
$m_n=\sup\{m\in\N\cup\{\infty\}\-w_n\in W_m\}$. We will construct a 
convergent subsequence of $(w_n)_{n\in\N}\sw\mathscr{W}$.

If $\{m_n\-n\in\N\}$ is finite (as an unordered set) then there exists 
$m\in\N\cup\{\infty\}$ such that $w_n\in W_m$ eventually, in which case 
$(w_n)_{n\in\N}$ has a convergent subsequence by compactness of $W_m$. 
Alternatively, if $\{m_n\-n\in\N\}$ is infinite then, with slight abuse 
of notation, we may pass to a subsequence and assume that $m_n$ is strictly 
increasing to $\infty$.

By definition of the Hausdorff metric, since $W_n\to W$ there exists 
$(h_n)_{n\in\N}\sw W_\infty$ such that $d_{\mathscr{M}}(w_n,h_n)\to 0$ as 
$n\to\infty$. Since $W_\infty$ is compact, $(h_n)_{n\in\N}$ has a convergent 
subsequence, and with further slight abuse of notation we pass to this 
subsequence and assume that $h_n\to h\in W_\infty$ as $n\to\infty$. 
Then $d_{\mathscr{M}}(w_n,h)\to 0$ as $n\to\infty$, and the proof is complete.
\end{proof}

\bibliographystyle{plainnat}
\bibliography{confirmation}

\begin{thebibliography}{25}
\providecommand{\natexlab}[1]{#1}
\providecommand{\url}[1]{\texttt{#1}}
\expandafter\ifx\csname urlstyle\endcsname\relax
  \providecommand{\doi}[1]{doi: #1}\else
  \providecommand{\doi}{doi: \begingroup \urlstyle{rm}\Url}\fi

\bibitem[Arratia(1979)]{arratia:1979}
R~A Arratia.
\newblock \emph{Coalescing {B}rownian motions on the line}.
\newblock University of Wisconsin--Madison, 1979.

\bibitem[Barton et~al.(2010)Barton, Etheridge, and V\'eber]{BEV2010}
N~H Barton, A~M Etheridge, and A~V\'eber.
\newblock A new model for evolution in a spatial continuum.
\newblock \emph{Electron. J. Probab.}, 15:\penalty0 162--216, 2010.

\bibitem[Barton et~al.(2013)Barton, Etheridge, and V{\'e}ber]{BEV2013}
N~H Barton, A~M Etheridge, and A~V{\'e}ber.
\newblock Modelling evolution in a spatial continuum.
\newblock \emph{Journal of Statistical Mechanics: Theory and Experiment},
  2013\penalty0 (01):\penalty0 P01002, 2013.

\bibitem[Berestycki et~al.(2013)Berestycki, Etheridge, and
  V\'eber]{berestycki/etheridge/veber:2013}
N~Berestycki, A~M Etheridge, and A~V\'eber.
\newblock Large-scale behaviour of the spatial {$\Lambda$-Fleming-Viot}
  process.
\newblock \emph{Ann. Inst. H. Poincar\'e}, 49\penalty0 (2):\penalty0 374--401,
  2013.

\bibitem[Billingsley(1995)]{billingsley:1995}
P~Billingsley.
\newblock \emph{{Probability and Measure}}.
\newblock Wiley, 1995.

\bibitem[Etheridge et~al.(2015)Etheridge, Wang, and Yu]{etheridge/wang/yu:2015}
A~Etheridge, S~Wang, and F~Yu.
\newblock Conditioning the logistic branching process on non-extinction.
\newblock \emph{Preprint}, 2015.

\bibitem[Etheridge et~al.(2016)Etheridge, Freeman, Penington, and
  Straulino]{EFP2016}
A~Etheridge, N~Freeman, S~Penington, and D~Straulino.
\newblock {B}ranching {B}rownian motion and {S}election in the {S}patial
  {L}ambda-{F}leming-{V}iot process.
\newblock \emph{to appear in Ann. Appl. Probab., arXiv:1512.03766}, 2016.

\bibitem[Etheridge(2008)]{E2008}
A~M Etheridge.
\newblock Drift, draft and structure: some mathematical models of evolution.
\newblock \emph{Banach Center Publ.}, 80:\penalty0 121--144, 2008.

\bibitem[Etheridge and Kurtz(2014)]{EK2014}
A~M Etheridge and T~G Kurtz.
\newblock {Genealogical constructions of population models}.
\newblock \emph{arXiv preprint arXiv:1402.6724}, 2014.

\bibitem[Etheridge et~al.(2014)Etheridge, V\'eber, and Yu]{EVY2014}
A~M Etheridge, A~V\'eber, and F~Yu.
\newblock Rescaling limits of the spatial {Lambda-Fleming-Viot} process with
  selection.
\newblock \emph{arXiv preprint arXiv:1406.5884}, 2014.

\bibitem[Ethier and Kurtz(1986)]{ethier/kurtz:1986}
S~N Ethier and T~G Kurtz.
\newblock \emph{{Markov processes: characterization and convergence}}.
\newblock Wiley, 1986.

\bibitem[Ferrari et~al.(2003)Ferrari, Fontes, and Wu]{FFW2003}
P~A Ferrari, L~R~G Fontes, and X-Y Wu.
\newblock Two-dimensional poisson trees converge to the {B}rownian web.
\newblock \emph{arXiv preprint math/0304247}, 2003.

\bibitem[Ferrari et~al.(2005)Ferrari, Fontes, and Wu]{FFW2005}
P~A Ferrari, L~R~G Fontes, and X-Y Wu.
\newblock Two-dimensional poisson trees converge to the {B}rownian web.
\newblock \emph{Annales de l'Institut Henri Poincar\'e (B) Probability and
  Statistics}, 41\penalty0 (5):\penalty0 851--858, 2005.

\bibitem[Ferrari et~al.(2004)Ferrari, Landim, and
  Thorisson]{ferrari/landim/thorisson:2004}
PA~Ferrari, C~Landim, and H~Thorisson.
\newblock Poisson trees, succession lines and coalescing random walks.
\newblock \emph{Annales de l'Institut Henri Poincar\'e (B) Probability and
  Statistics}, 40\penalty0 (2):\penalty0 141--152, 2004.

\bibitem[Fontes et~al.(2004)Fontes, Isopi, Newman, and Ravishankar]{FIN2004}
L~R~G Fontes, M~Isopi, C~M Newman, and K~Ravishankar.
\newblock The {B}rownian web: characterization and convergence.
\newblock \emph{The Annals of Probability}, 32\penalty0 (4):\penalty0
  2857--2883, 2004.

\bibitem[Krone and Neuhauser(1997)]{krone/neuhauser:1997}
S~M Krone and C~Neuhauser.
\newblock {Ancestral processes with selection}.
\newblock \emph{Theor. Pop. Biol.}, 51:\penalty0 210--237, 1997.

\bibitem[Micaux(2007)]{micaux:2007}
B~Micaux.
\newblock Flots stochastiques d’op\'erateurs dirig\'es par des bruits
  gaussiens et poissonniens. 2007.
\newblock \emph{Th\`ese pr\'esent\'ee pour obtenir le grade de docteur en
  sciences de l’universit\'e Paris XI, sp\'ecialit\'e math\'ematiques}, 2007.

\bibitem[Neuhauser and Krone(1997)]{neuhauser/krone:1997}
C~Neuhauser and S~M Krone.
\newblock {Genealogies of samples in models with selection}.
\newblock \emph{Genetics}, 145:\penalty0 519--534, 1997.

\bibitem[Newman et~al.(2005)Newman, Ravishankar, and
  Sun]{newman/ravishankar/sun:2005}
C~M Newman, K~Ravishankar, and R~Sun.
\newblock Convergence of coalescing nonsimple random walks to the {B}rownian
  web.
\newblock \emph{Electron. J. Probab.}, 10:\penalty0 21--60, 2005.

\bibitem[Newman et~al.(2015)Newman, Ravishankar, and
  Schertzer]{newman/ravishankar/schertzer:2015}
C~M Newman, K~Ravishankar, and E~Schertzer.
\newblock {B}rownian net with killing.
\newblock \emph{Stoch. Proc. Appl.}, 125\penalty0 (3):\penalty0 1148--1194,
  2015.

\bibitem[Schertzer et~al.(2014)Schertzer, Sun, and Swart]{SSS2014}
E~Schertzer, R~Sun, and J~M Swart.
\newblock Stochastic flows in the {B}rownian web and net.
\newblock \emph{Memoirs Amer. Math. Soc.}, 227, 2014.

\bibitem[Schertzer et~al.(2015)Schertzer, Sun, and Swart]{SSS2015}
E~Schertzer, R~Sun, and J~M Swart.
\newblock The {B}rownian web, the {B}rownian net, and their universality.
\newblock \emph{arXiv:1506.00724}, pages 1--80, 2015.

\bibitem[Straulino(2014)]{straulino:2014}
D~Straulino.
\newblock \emph{Selection in a spatially structured population}.
\newblock DPhil thesis, University of Oxford, 2014.

\bibitem[Sun and Swart(2008)]{SS2008}
R~Sun and J~M Swart.
\newblock The {B}rownian net.
\newblock \emph{The Annals of Probability}, pages 1153--1208, 2008.

\bibitem[V\'eber and Wakolbinger(2013)]{VW2013}
A~V\'eber and A~Wakolbinger.
\newblock The spatial {L}ambda-{F}leming-{V}iot process: an event-based
  construction and a lookdown representation.
\newblock \emph{Ann. Inst. H. Poincar\'e Probab. Statist. (to appear)},
  arXiv:1212.5909v2, 2013.

\end{thebibliography}

\end{document}